\documentclass[12pt]{article}
\title{Large deviations of the free energy in the O'Connell-Yor polymer}
\author{Chris Janjigian\footnote{Department of Mathematics, University of Wisconsin - Madison, \href{mailto:janjigia@math.wisc.edu}{janjigia@math.wisc.edu}.}}
\usepackage{amsfonts}
\usepackage{amsmath}
\usepackage{amsthm,color}
\usepackage{amssymb}
\usepackage{hyperref}
\usepackage[title]{appendix}

\theoremstyle{plain}
    \newtheorem{theorem}{Theorem}
	\numberwithin{theorem}{section}
    \newtheorem{lemma}[theorem]{Lemma}
    \newtheorem{proposition}[theorem]{Proposition}
    \newtheorem{corollary}[theorem]{Corollary}

\theoremstyle{definition}
    
    \newtheorem{remark}[theorem]{Remark}

\theoremstyle{remark}

\begin{document}
\maketitle

\begin{abstract}
We investigate large deviations of the free energy in the O'Connell-Yor polymer through a variational representation of the positive real moment Lyapunov exponents of the associated parabolic Anderson model. Our methods yield an exact formula for all real moment Lyapunov exponents of the parabolic Anderson model and a dual representation of the large deviation rate function with normalization $n$ for the free energy.
\end{abstract}

\section{Introduction}

This paper studies the model of a 1+1 dimensional semi-discrete directed polymer in a random environment due to O'Connell and Yor \cite{OY01}. Directed polymers in random environments were introduced in the statistical physics literature in \cite{HH85}, with mathematical work following in \cite{IS88} and \cite{Bo89}. A physically motivated mathematical introduction to this family of models can be found in the survey article \cite{CSY04}. These models have attracted substantial attention recently, in part due to the conjecture that under suitable regularity assumptions they lie in the KPZ universality class. For a discussion of this conjecture and further references, we refer the reader to \cite{Co12}.

Our study will focus on the point-to-point polymer partition function of this model, which \cite{OY01} defines for $n \in \mathbb{N}$ and $\beta > 0$ as
\begin{align}\label{OYpoly}
Z_n(\beta) &= \int_{0 < s_1 < \dots < s_{n-1} < n}\exp\left[\beta \left(B_1(0,s_1) + \dots + B_n(s_{n-1},n) \right) \right]ds_1 \dots ds_{n-1},
\end{align}
where $\{B_i\}_{i=1}^\infty$ is a family of i.i.d. standard Brownian motions. This partition function is the normalizing constant for a quenched polymer measure on non-decreasing c\`adl\`ag paths $f:\mathbb{R}_+ \to \mathbb{N}$ with $f(0) = 1$ and $f(n) = n$. Up to a constant factor, $Z_n(\beta)$ is the conditional expectation of a functional of a Poisson path on the event that the path is at $n$ at time $n$. More precisely, let $\pi(\cdot)$ be a unit rate Poisson process on $\mathbb{R}_+$ which is independent of the family $\{B_i\}_{i=1}^\infty$ and denote by $\mathcal{E}$ the expectation with respect to the law of this Poisson process. With the notation 
\begin{align}
A_{n,t} = \{s_1, \dots s_{n-1} : 0 < s_1 < \dots < s_{n-1} < t \}, \label{Weyldef}
\end{align} 
we have
\begin{align}
\mathcal{E}_{\pi(0)=1} \left[ e^{\int_0^n \beta dB_{\pi(s)}(s)} | \pi(n) =n \right] &= |A_{n,n}|^{-1} Z_n(\beta). \label{condexp}
\end{align}
The prefactor of $|A_{n,n}|^{-1}$ accounts for the fact that the ordered jump points of $\pi$ on $[0,n]$ conditioned on $\pi(n) = n$ are uniformly distributed on the Weyl chamber $A_{n,n}$.

The O'Connell-Yor polymer model was originally introduced in \cite{OY01} in connection with a generalization of the Brownian queueing model. Based on the work of Matsumoto and Yor \cite{MY01}, O'Connell and Yor were able to show the existence of a stationary version of this model satisfying an analogue of Burke's theorem for M/M/1 queues. The Burke property makes the O'Connell-Yor polymer one of the four polymer models considered exactly solvable, the others being the continuum directed polymer studied in \cite{ACQ11}, the log-gamma polymer introduced by Sepp\"al\"ainen in \cite{Sep12}, and the strict-weak gamma polymer studied in \cite{CSS14} and \cite{OO14}. A precise statement of the Burke property for this model and an outline of how this property leads to the main result of this paper is given in subsection \ref{subsecstat}. Subsequent work on the representation theoretic underpinnings of the exact solvability of these models can be found in the work of Borodin and Corwin on Macdonald processes \cite{BC14a} and the work of O'Connell connecting this model to the quantum Toda lattice \cite{O12}.

As one of the few tractable models in the KPZ universality class, this polymer model has been extensively studied: Moriarty and O'Connell \cite{MO07} rigorously computed the free energy; Sepp\"al\"ainen and Valk\'o \cite{SV10} identified the scaling exponents; Borodin, Corwin, and Ferrari \cite{BCF14} showed that the model lies in the KPZ universality class by proving the Tracy-Widom limit for the free energy fluctuations; and Borodin and Corwin \cite{BC14b} proved a contour integral representation for the integer moments and computed their large $n$ asymptotics.

The main results of this paper are Theorems \ref{thm:MLE} and \ref{thm:LDP}, which compute the moment Lyapunov exponents and large deviation rate function with normalization $n$ for the free energy respectively. Theorem \ref{thm:MLE} can be thought of as an extension of the asymptotics studied in \cite{BC14b}. There, the authors use a contour integral representation for the integer moments of $Z_n(\beta)$ to compute the limit
\begin{align*}
\lim_{n \to \infty} \frac{1}{n} \log E\left[Z_{n}(\beta)^k\right]
\end{align*}
for any $k \in \mathbb{N}$. In \cite[Appendix A.1]{BC14b}, as part of a replica computation of the free energy, they conjecture an analytic continuation of their formula to $k >0$. We are able to compute the above limit for all $k \in \mathbb{R}$, confirming this conjecture.

The proofs of Theorems \ref{thm:MLE} and \ref{thm:LDP} follow an approach introduced by Sepp\"al\"ainen in \cite{Sep98}. Georgiou and Sepp\"al\"ainen used this method to compute the large deviation rate function with normalization $n$ for the free energy in the log-gamma polymer in \cite{GS13}. The key technical condition making this scheme tractable is the independence provided by the Burke property, which the log-gamma polymer shares with the O'Connell-Yor polymer. It is therefore natural to expect that the techniques of \cite{GS13} should also apply in this setting; we take this as our starting point.

Physically, we can view the parameter $\beta$ as an inverse temperature. In this sense, we can think of directed polymer models as positive temperature analogues of directed percolation. This leads to a natural coupling between the directed polymer and directed percolation models, which we take advantage of throughout the paper. The directed percolation model associated to the O'Connell-Yor polymer is Brownian directed percolation; see \cite{OY01} and \cite{HMO02} for a discussion of this model. A distributional equivalence between the last passage time in Brownian directed percolation and the largest eigenvalue of a standard GUE matrix was discovered independently by Baryshnikov  \cite[Theorem 0.7]{Bar01} and Gravner, Tracy, and Widom \cite{GTW01}, both in 2001. The known large deviations \cite{Ledoux} for top eigenvalues of GUE matrices give useful estimates in several of the proofs that follow. The precise results we use are collected in subsection \hyperref[GUEsub]{\ref*{GUEsub}} of the appendix.

This polymer and the log-gamma polymer are the only positive temperature polymer models for which precise large deviations have been studied. Precise left tail large deviations, which have non-universal scalings \cite{B09}, remain open for both models. The lattice Gaussian directed percolation and polymer models in one spatial dimension have left tail large deviations with normalization $\frac{n^2}{\log(n)}$ \cite[Theorem 1.2]{CGM09} \cite[p. 774]{B09}, while the Brownian directed percolation model has left tail large deviations with normalization $n^2$ \cite[(1.26)]{Ledoux}. It is natural to expect that the left tail large deviations for this model should follow those of Brownian directed percolation rather than the Gaussian lattice models, but we do not currently have a proof of this result.

 \textbf{Acknowledgements.} The author would like to thank Timo Sepp\"al\"ainen for suggesting this problem; Benedek Valk\'o for many helpful conversations; and the anonymous referees for their feedback, which helped to improve the presentation of the paper. The author also benefited greatly from attending Ivan Corwin's summer course on KPZ at MSRI.

\textbf{Notation.} When $b(\cdot)$ is Brownian motion and $s \leq t$, we adopt the convention $b(s,t) = b(t) - b(s)$. For $a,b \in [-\infty,\infty]$, we set $a \vee b = \max(a,b)$ and $a \wedge b = \min(a,b)$. The polygamma functions are denoted by $\Psi_k(x)$, where $\Psi_0(x) = \frac{d}{dx}\log \Gamma(x)$ and $\Psi_n(x) = \frac{d}{dx} \Psi_{n-1}(x)$.

For $f,g:\mathbb{R} \to (-\infty,\infty]$, the Legendre-Fenchel transform is defined by $f^*(\xi) = \sup_{x \in \mathbb{R}}\{x \xi - f(x)\}$ and the infimal convolution is defined by $f \square g(x) = \inf_{y \in \mathbb{R}}\{f(x-y) + g(y)\}$. For properties of these operators, we refer the reader to \cite{Rock}.

A random variable has a  $\Gamma(\theta,1)$ distribution if it has density $\Gamma(\theta)^{-1}x^{\theta - 1}e^{-x}1_{\{x>0\}}$ with respect to the Lebesgue measure on $\mathbb{R}$.

\section{Preliminaries and statement of results}
\subsection{Definition of the polymer model and statement of results}
Let $\{B_i\}_{i=0}^\infty$ be a family of independent two-sided standard Brownian motions. Define partition functions for $j,n \in \mathbb{Z}_+$ with $j < n$ and $s,t \in (0,\infty)$ with $s < t$ by
\begin{align}\label{polydefgr}
Z_{j,n}(u,t) &= \int\displaylimits_{u < u_j < \dots < u_{n-1} < t} e^{B_j(u,u_j) + \sum_{i= j+1}^{n-1} B_i (u_{i-1}, u_i) + B_{n}(u_{n-1},t)}du_j \dots du_{n-1}.
\end{align}
For the case $j = n$, we define
\begin{align}\label{polydefeq}
Z_{j,j}(u,t) &= e^{B_j(u,t)}.
\end{align}
We will refer to the $j,n$ variables as space and the $u,t$ variables as time. Translation invariance of Brownian motion and our assumption that the environment is i.i.d. immediately imply that the distribution of the partition function is shift invariant. 

It follows from Brownian scaling that for $\beta > 0$ and $n > 1$ we have
\begin{align*}
Z_n(\beta) \stackrel{\tiny{d}}{=} \beta^{-2(n-1)}Z_{1,n}(0,\beta^2 n).
\end{align*}
For the remainder of the paper, we will only consider partition functions of the form $Z_{j,n}(u,t)$; results for these partition functions can be translated into results for $Z_n(\beta)$ using this distributional identity.

Next, we argue that the partition function is supermultiplicative: that is, for $j, n,m \in \mathbb{Z}_+$, $v \geq 0$, and $u,t > 0$ ,
\begin{align}\label{gensubmult}
Z_{j,j+n+m}(v,v+t+u) &\geq Z_{j,j+n}(v,v+t)Z_{j+n,j+n+m}(v+t,v+t+u).
\end{align}
For notational convenience, we will consider the case $j= v = 0$.  For $t,u > 0$, we have
\begin{align*}
Z_{0,0}(0,t+u) = e^{B_0(0,t+u)} = e^{B_0(0,t)} e^{B_0(t,t+u)} = Z_{0,0}(0,t)Z_{0,0}(t,t+u).
\end{align*}
For $m,n \in \mathbb{N}$ and $t,u > 0$, we have
\begin{align}\label{submult}
Z_{0,n+m}(0,t+u) &= \int\displaylimits_{0 < u_0 < \dots < u_{n+m-1} < t+u} e^{B_0(0,u_0) + \sum_{i= 1}^{n+m-1} B_i (u_{i-1}, u_i) + B_{n+m}(u_{n+m-1},t + u)}du_0 \dots du_{n+m-1} \notag \\
&\geq \int\displaylimits_{\stackrel{0 < u_0 < \dots < u_{n+m-1} < t+u}{u_{n-1} < t < u_n}} e^{B_0(0,u_0) + \sum_{i= 1}^{n+m-1} B_i (u_{i-1}, u_i) + B_{n+m}(u_{n+m-1},t+u)}du_0 \dots du_{n+m-1} \notag \\
&= Z_{0,n}(0,t)Z_{n,n+m}(t,t+u).
\end{align}
When $m < n$ and $t,u > 0$, we decompose $\log Z_{m,n}(u,t)$ as follows:
\begin{align}
\log Z_{m,n}(u,t) &= B_{n}(t) - B_m(u) + \log C_{m, n}(u,t) \label{semimg}
\end{align}
where
\begin{align*}
C_{m,n}(u,t) &= \int_u^t \int_{u_{m}}^t \dots \int_{u_{n-2}}^t e^{B_m(u_m) + \sum_{i=m+1}^{n-1} B_i (u_{i-1}, u_i) - B_{n}(u_{n-1})}du_{n-1} \dots du_{m+1} du_m
\end{align*}
is strictly increasing in $t$ and strictly decreasing in $u$. It follows that
\begin{align*}
Z_{m,n}(0,t+u) &\geq Z_{m,n}(0,t)Z_{n,n}(t,t+u), \\
Z_{m,n}(0,t+u) &\geq Z_{m,m}(0,t)Z_{m,n}(t,t+u).
\end{align*}

The free energy for \hyperref[OYpoly]{(\ref*{OYpoly})} was computed in \cite{MO07}. We mention that, as in \cite[Lemma 4.1]{GS13}, once one knows the existence and continuity of the free energy, a variational problem similar to the one we study for the rate function in this paper can be used to compute the value of the free energy. We have
\begin{lemma} \cite{MO07} \label{lem:fe} Fix $s,t \in (0,\infty)$. Then the almost sure limit
\begin{align*}
\rho(s,t) &= \lim_{n \to \infty} \frac{1}{n} \log Z_{1, \lfloor ns \rfloor}(0,nt)
\end{align*}
exists and is given by
\begin{align*}
\rho(s,t) &= \min_{\theta > 0} \left \{ \theta t - s \Psi_0(\theta) \right \} = t \Psi_1^{-1}\left(\frac{t}{s}\right) - s \Psi_0\left(\Psi_1^{-1}\left(\frac{t}{s}\right)\right).
\end{align*} 
\end{lemma}

The main result of this paper is a computation of the real moment Lyapunov exponents of the parabolic Anderson model associated to \hyperref[polydefgr]{(\ref*{polydefgr})} and, through an application of the G\"artner-Ellis theorem, the large deviation rate function with normalization $n$ for the free energy of the polymer. Specifically, we have
\begin{theorem} \label{thm:MLE}
Let $s,t \in (0,\infty)$ and $\xi \in \mathbb{R}$. Then
\begin{align*} 
\Lambda_{s,t}(\xi) &= \lim_{n \to \infty} \frac{1}{n} \log E\left[e^{\xi \log Z_{1,\lfloor ns \rfloor}(0,nt)}\right] \\
&= \begin{cases}
\xi \rho(s,t) & \xi \leq 0 \\
\displaystyle\min_{\mu > 0} \left\{ t \left(\frac{\xi^2}{2} + \xi \mu \right) - s \log \frac{\Gamma(\mu + \xi)}{\Gamma(\mu)} \right\} & \xi >0
\end{cases}
\end{align*}
and $\Lambda_{s,t}(\xi)$ is a differentiable function of $\xi \in \mathbb{R}$.
\end{theorem}
\begin{theorem}\label{thm:LDP}
Fix $s,t \in (0,\infty)$. The distributions of $n^{-1}\log Z_{1,\lfloor ns \rfloor}(0,nt)$ satisfy a large deviation principle with normalization $n$ and convex good rate function
\begin{align*}
I_{s,t} (x) &= \begin{cases}
\infty & x < \rho(s,t) \\
\Lambda_{s,t}^*(x) & x \geq \rho(s,t)
\end{cases}.
\end{align*}
\end{theorem}
\begin{remark} The function being minimized for $\xi > 0$ in \hyperref[thm:MLE]{Theorem \ref*{thm:MLE}} is convex and coercive, so the minimum is attained. The details this fact are worked out in the proof of \hyperref[cor:lfform]{Corollary \ref*{cor:lfform}}.
\end{remark}
\subsection{Definition of the stationary model and proof outline} \label{subsecstat}
Let $B(t)$ be a two-sided Brownian motion independent of the family $\{B_i\}_{i=0}^\infty$ and for $\theta>0$, $t \in \mathbb{R}$ and $n \in \mathbb{Z}_+$ define point-to-point partition functions by
\begin{align*}
Z_n^\theta(t) &= \int\displaylimits_{- \infty < u_0 < u_1 < \dots < u_{n-1} < t}e^{ \theta u_0 - B(u_0) + B_1(u_0, u_1) + \dots + B_{n}(u_{n-1},t)} du_0 \dots du_{n-1},
\end{align*}
with the convention that
\begin{align*}
Z_0^\theta (t) &= e^{\theta t - B(t)}.
\end{align*}
We can think of $Z_n^\theta(t)$ as a modification of the polymer in the previous subsection where we add a spatial dimension and start in the infinite past.  For $s,t > 0$ and $n$ sufficiently large that $ns\geq1$, we obtain a decomposition of $Z_{\lfloor ns \rfloor}^\theta(nt)$ into terms that involve the partition functions we are studying by considering where paths leave the potential of the Brownan motion $B$:
\begin{align}
Z_{\lfloor n s \rfloor}^\theta(nt)&= \int_0^{nt} Z_0^\theta(u)Z_{1,\lfloor n s \rfloor}(u,nt)du + \sum_{j=1}^{\lfloor ns \rfloor} Z_j^\theta(0)Z_{j, \lfloor n s \rfloor}(0,nt). \label{coupling}
\end{align}
This expression also leads to the interpretation of $Z_n^\theta(t)$ as a modification of the point-to-point partition function discussed in the previous subsection where we have added boundary conditions.

We will refer to this model as the stationary polymer, where the term stationary comes from the fact that it satisfies an analogue of Burke's theorem for M/M/1 queues. This fact is one of the main contributions of \cite{OY01} and we refer the reader to that paper for a more in depth discussion of the connections to queueing theory. We follow the notation of $\cite{SV10}$, which contains the version of the Burke property that will be used in this paper. Define $Y_0^\theta(t) = B(t)$ and for $k \geq 1$ recursively set
\begin{align}
r_k^\theta(t) &= \log \int_{-\infty}^t e^{Y_{k-1}^\theta(u,t) - \theta(t-u) + B_k(u,t)} du, \notag \\
Y_k^\theta(t) &= Y_{k-1}^\theta(t) + r_k^\theta(0) - r_k^\theta(t), \label{queuedef} \\ 
X_k^\theta(t) &= B_k(t) + r_k^\theta(0) - r_k^\theta(t); \notag
\end{align}
then we have
\begin{lemma} \cite[Theorem 3.3]{SV10} Let $n \in \mathbb{N}$ and $0 \leq s_n \leq s_{n-1} \leq \dots \leq s_1 < \infty$. Then over $j$, the following random variables and processes are all mutually independent.
\begin{align*}
& r_j(s_j) \ \text{ and } \    \{X_j(s): s\leq s_j\} \ \text{ for } \    1\leq j\leq n, \quad  \{Y_n(s): s\leq s_n\}, \quad \\
& \qquad  \text{and} \quad
\{Y_j(s_{j+1},s): s_{j+1}\leq s\le s_j\} \ \text{ for }   1\leq j\leq n-1.
 \end{align*}
Furthermore,  the $X_j$ and $Y_j$  processes are standard Brownian motions,  and $e^{-r_j(s_j)}$ is $\Gamma(\theta,1)$ distributed.
\end{lemma} 

An induction argument shows that
\begin{align}
\sum_{k=1}^n r_k^\theta(t) &= B(t) - \theta t + \log Z_n^\theta (t). \label{statdecomp}
\end{align}

As we will see shortly, expression \hyperref[coupling]{(\ref*{coupling})} would lead to a variational formula for the right tail rate function we are looking for in terms of the right tail rate function of $Z_{\lfloor ns \rfloor}^\theta(nt)$. This right tail rate function would be tractable using \hyperref[statdecomp]{(\ref*{statdecomp})} if $B(nt)$ were independent of $\sum_{k=1}^{\lfloor ns \rfloor} r_k^\theta(nt)$; as this is not the case, it is convenient to rewrite \hyperref[coupling]{(\ref*{coupling})} in a form that separates these two terms:
\begin{align}
e^{\sum_{k=1}^{\lfloor ns \rfloor} r_k^\theta(nt)} &= n \int_0^t \frac{Z_0^\theta(nu)}{Z_0^\theta(nt)}Z_{1, \lfloor ns \rfloor}(nu,nt) du + \sum_{j=1}^{\lfloor ns \rfloor} \frac{Z_j^\theta(0)}{Z_0^\theta(nt)} Z_{j,\lfloor ns \rfloor}(0,nt) \label{LDcoupling}.
\end{align}

We now briefly outline the proof of \hyperref[thm:MLE]{Theorem \ref*{thm:MLE}}. In order to compute the positive moment Lyapunov exponents, we consider the dual problem of establishing right tail large deviations for the free energy. As is typically the case, existence and regularity of the right tail rate function follows from subadditivity arguments. It then follows from \hyperref[LDcoupling]{(\ref*{LDcoupling})} that this right tail rate function solves a variational problem in terms of computable rate functions coming from the stationary model. Taking Legendre-Fenchel transforms brings us back to the study of moment Lyapunov exponents and gives the variational problem a linear structure which makes it tractable.

For non-positive exponents, we use crude estimates on the partition function to identify the limit. We are able to do this because the left tail large deviations for the free energy are are strictly subexponential while the moment Lyapunov exponents are only sensitive to exponential scale large deviation.

\section{A variational problem for the right tail rate function}
\subsection{Definitions and notation}
The goal of this subsection is to introduce the right tail rate function for the free energy, which we will denote $J_{s,t}(x)$, and the rate functions coming from the stationary model which appear in the variational expression for $J_{s,t}(x)$. We will defer some of the proofs of technical results about the existence and regularity of these rate functions to \hyperref[sec:existreg]{Appendix \ref*{sec:existreg}}. We begin by defining these functions and addressing existence.
\begin{theorem}\label{thm:exist}
For all $s\geq 0$, $t > 0$ and $r \in \mathbb{R}$, the limit
\begin{align*}
J_{s,t}(r) &= \lim_{n \to \infty} - \frac{1}{n} \log P \left(\log Z_{1, \lfloor ns \rfloor}(0, nt) \geq nr \right)
\end{align*}
exists and is $\mathbb{R}_+$ valued.  Moreover, $J_{s,t}(r)$ is continuous, convex, subadditive, and positively homogeneous of degree one as a function of $(s, t,r) \in [0,\infty) \times(0,\infty) \times\mathbb{R}$. For fixed $s$ and $t$, $J_{s,t}(r)$ is increasing in $r$ and $J_{s,t}(r) = 0$ if $r \leq \rho(s,t)$.
\end{theorem}
\noindent The proof of  \hyperref[thm:exist]{Theorem \ref*{thm:exist}} can be found in \hyperref[subsec:existreg]{subsection \ref*{subsec:existreg}} of \hyperref[sec:existreg]{Appendix \ref*{sec:existreg}}. 

Next, we define the computable rate functions from the stationary model.  By the Burke property for the stationary model, the first limit below can be computed as the right branch of a Cram\'er rate function. For $s,t > 0$, we set
\begin{align*}
U_s^{\theta}(x) &= - \lim \frac{1}{n} \log P \left(\sum_{k=1}^{\lfloor ns \rfloor} r_k^\theta (0)\geq nx \right) \\
&= \begin{cases}
0 & x \leq - s \Psi_0(\theta) \\
x(\theta - \Psi_0^{-1}(-\frac{x}{s})) + s \log\frac{\Gamma(\theta)}{\Gamma(\Psi_0^{-1}(-\frac{x}{s}))} & x > - s\Psi_0(\theta)
\end{cases},\\
R_t^\theta(x) &= - \lim \frac{1}{n} \log P \left( B(nt) - \theta n t \geq nx \right) = \begin{cases}
0 & x \leq -\theta t \\
\frac{1}{2}\left(\frac{x+ \theta t}{\sqrt{t}}\right)^2 & x > -\theta t 
\end{cases}.
\end{align*}
We may continuously extend $U_s^\theta(x)$ to $s = 0$ by setting
\begin{align*}
U_0^\theta(x) &= \begin{cases}
0 & x \leq 0 \\
x \theta & x > 0
\end{cases}.
\end{align*}
We record the Legendre-Fenchel transforms of these functions below:
\begin{align*}
(U_s^\theta)^*(\xi) &= \begin{cases}
\infty & \xi < 0 \text{ or } \xi \geq \theta \\
s \log \frac{\Gamma(\theta - \xi)}{\Gamma(\theta)} & 0 \leq \xi < \theta
\end{cases}, \qquad 
(R_t^\theta)^*(\xi) = \begin{cases}
\infty & \xi < 0 \\
t(\frac{\xi^2}{2} - \theta \xi) & \xi \geq 0
\end{cases}.
\end{align*}

The next lemma implies existence of the rate functions which will appear when we use equation \hyperref[LDcoupling]{(\ref*{LDcoupling})} to prove that $J_{s,t}(x)$ satisfies a variational problem in the next subsection. Versions of this result appear in several other papers, so we elect not to re-prove it. The exact statement we need appears in  \cite{GS13}.
\begin{lemma}\label{lem:indep}
\cite[Lemma 3.6]{GS13}) Suppose that for each $n$, $X_n$ and $Y_n$ are independent, that the limits
\begin{align*}
\lambda(s) &= \lim_{n \to \infty} - \frac{1}{n} \log P \left( X_n \geq  ns  \right), \qquad \phi(s) = \lim_{n \to \infty} - \frac{1}{n} \log P \left( Y_n \geq n s \right)
\end{align*}
exist, and that $\lambda$ is continuous. If there exists $a_\lambda$ and $a_\phi$ so that $\lambda(a_\lambda) = \phi(a_\phi) = 0$, then 
\begin{align*}
\lim_{n \to \infty} - \frac{1}{n} \log P\left(X_n + Y_n \geq nr\right) &= \begin{cases}
 \inf_{a_\lambda \leq s \leq r - a_\phi}\{\phi(r-s) + \lambda(s)\} & r\geq a_\phi + a_\lambda \\
0 & r \leq a_\phi + a_\lambda
\end{cases} \\
 &= \lambda \square \phi(r)
\end{align*}
\end{lemma}

We define rate functions corresponding to the two parts of the decomposition in \hyperref[LDcoupling]{(\ref*{LDcoupling})} as follows: for $a \in [0,t)$, $u\in (0,s]$, $v \in [0,s)$, and $x \in \mathbb{R}$ set
\begin{align}
G_{a,s,t}^\theta(x) &= - \lim \frac{1}{n} \log P \left( B(na,nt) - \theta n (t-a) + \log Z_{1,\lfloor n s \rfloor}(na,nt) \geq nx \right), \notag \\
H_{u, v, s, t}^\theta(x) &= - \lim \frac{1}{n} \log P \left(- \log Z_0^\theta (nt) + \log Z_{\lfloor n u\rfloor}^\theta(0) + \log Z_{\lfloor nv \rfloor, \lfloor n s \rfloor}(0,nt) \geq nx \right). \label{GHdef}
\end{align}

Recall that $\log Z_j^\theta(0) = \sum_{k=1}^j r_k^\theta(0)$ is measurable with respect to the sigma algebra $\sigma(B(s), B_k(s) : 1 \leq k \leq j;  s \leq 0)$ and that for $0 \leq u < nt$, $\log Z_{j,\lfloor ns \rfloor}(u,t)$ is measurable with respect to the sigma algebra $\sigma(B_k(s_k) :  j \leq k \leq \lfloor ns\rfloor, u \leq s_j \leq nt)$. Combining the independence of the environment with the computations above, \hyperref[thm:exist]{Theorem \ref*{thm:exist}} and \hyperref[lem:indep]{Lemma \ref*{lem:indep}} imply that $G_{a,s,t}^\theta(x)$ and $H_{u, v, s, t}^\theta(x)$ are well-defined. In particular, we immediately obtain
\begin{corollary}\label{cor:infcon}
For $a\in [0,t)$ and $u \in (0,s]$, and $v \in [0,s)$
\begin{align*}
G_{a,s,t}^\theta(x) &= R_{t - a}^\theta \square J_{s,t - a}(x) , \qquad 
H_{u,v,s,t}^\theta(x) = R_t^\theta \square U_u^\theta \square J_{s - v, t}(x).
\end{align*}
\end{corollary}

In order to show that \hyperref[LDcoupling]{(\ref*{LDcoupling})} leads to a variational problem, we need some regularity on $G_{a,s,t}^\theta(x)$ and $H_{u,v,s,t}^\theta(x)$. The three results that follow are purely technical, so we defer their proofs to \hyperref[subsec:statreg]{subsection \ref*{subsec:statreg}} of \hyperref[sec:existreg]{Appendix \ref*{sec:existreg}}. \hyperref[lem:Hcont]{Lemma \ref*{lem:Hcont}} gives a strong kind of local uniform continuity of $H_{u,v,s,t}^\theta(x)$ and \hyperref[lem:Gcont]{Lemma \ref*{lem:Gcont}} gives the same for $G_{a,s,t}^\theta(x)$. The difference between the two statements comes from \hyperref[lem:Ginf]{Lemma \ref*{lem:Ginf}}, which shows that $G_{a,s,t}^\theta(x)$ degenerates to infinity locally uniformly near $a = t$.
\begin{lemma}\label{lem:Hcont}
Fix $\theta, s, t > 0$ and a compact set $K \subseteq \mathbb{R}$. Then
\begin{align*}
\lim_{\delta, \gamma, \epsilon \downarrow 0} \sup_{\stackrel{a,b,b' \in [0,s] : |b-b'|< \delta}{r_1, r_2 \in K : |r_1 - r_2|< \epsilon}}\left\{ |H_{a,b,s,t + \gamma}^\theta(r_1) - H_{a,b',s,t}^\theta(r_2)| \right\} = 0.
\end{align*}
\end{lemma}
\begin{lemma}\label{lem:Gcont}
Fix $\theta, s, t > 0$ and $0 < \delta \leq t$ and a compact set $K \subseteq \mathbb{R}$. Then
\begin{align*}
\lim_{\epsilon,\gamma \downarrow 0} \sup_{\stackrel{a_1,a_2 \in [0,t-\delta] : |a_1 - a_2| < \gamma}{ r_1, r_2 \in K : |r_1 - r_2|< \epsilon}}\left\{ |G_{a_1,s,t}^\theta(r_1) - G_{a_2,s,t}^\theta(r_2)|\right\} = 0.
\end{align*}
\end{lemma}
\begin{lemma}\label{lem:Ginf}
Fix $\theta, s, t > 0$ and $K \subset \mathbb{R}$ compact. Then
\begin{align*}
\liminf_{a \uparrow t} \inf_{x \in K} \left\{G_{a,s,t}^\theta(x)\right\} &= \infty.
\end{align*}
\end{lemma}

\subsection{Coarse graining and the variational problem} 
Fix  $a \in [0,t)$ and $0 < \delta \leq t - a$. Then \hyperref[LDcoupling]{(\ref*{LDcoupling})} implies the following lower bounds
\begin{align}
\log \left(n \int_{a}^{a + \delta} \frac{Z_0^\theta(nu)}{Z_0^\theta(nt)}Z_{1,\lfloor n s \rfloor}(nu,nt)du \right)&\leq \sum_{k=1}^{\lfloor n s \rfloor} r_k^\theta(nt), \label{ineq:lowerbd1}  \\
-\log Z_0^\theta (nt) + \log Z_j^\theta(0) + \log Z_{j,\lfloor n s \rfloor}(0,nt) &\leq \sum_{k=1}^{\lfloor n s \rfloor} r_k^\theta(nt). \label{ineq:lowerbd2} 
\end{align}
For any partition $\{a_i\}_{i=0}^{N}$ of $[0,t]$, we also have
\begin{align}
\sum_{k=1}^{\lfloor ns \rfloor} r_k^\theta(nt) \leq & \max_{0 \leq i \leq N-1}\left\{ \log \left( n \int_{a_i}^{a_{i+1}}  \frac{Z_0^\theta(nu)}{Z_0^\theta(nt)}Z_{1,\lfloor n s \rfloor}(nu,nt)du \right)\right\} \notag \\
&\vee \max_{1 \leq j \leq \lfloor ns \rfloor} \left\{ - \log Z_0^\theta (nt) + \log Z_j^\theta(0)+ \log Z_{j,\lfloor n s \rfloor}(0,nt)\right\} + \log( N + 1 + ns). \label{ineq:upperbd1}
\end{align}
Our goal is now to show that estimates \hyperref[ineq:lowerbd1]{(\ref*{ineq:lowerbd1})}, \hyperref[ineq:lowerbd2]{(\ref*{ineq:lowerbd2})}, and \hyperref[ineq:upperbd1]{(\ref*{ineq:upperbd1})} above lead to a variational characterization of the right tail rate function $J_{s,t}(x)$:
\begin{align}
U_s^\theta(x) &= \min \{ \inf_{0 \leq a < t} \left\{ G_{a,s,t}^\theta(x) \right\}, \inf_{0 \leq a < s} \left\{H_{a, a,s,t}^\theta(x)\right\} \} \notag \\
&= \min\{\inf_{0 \leq a < t} \left\{ R_{t - a}^\theta \square J_{s,t - a}(x) \right\}, \inf_{0 \leq a < s} \left\{R_t^\theta \square U_a^\theta \square J_{s - a, t}(x)\right\}\}. \label{Uvarprob}
\end{align}
To improve the presentation of the paper, we have moved some of the estimates in the proofs that follow to \hyperref[sec:estimates]{Appendix \ref*{sec:estimates}}.
\begin{lemma}\label{lem:varprob1}
Fix $\theta > 0$, $(s,t) \in (0,\infty)^2$ and $x \in \mathbb{R}$. Then
\begin{align*}
U_s^\theta(x) &\leq \min \{ \inf_{0 \leq a < t}\left\{ G_{a,s,t}^\theta(x)\right\}, \inf_{0 \leq a < s} \left\{ H_{a, a,s,t}^\theta(x) \right\} \}.
\end{align*}
\end{lemma}
\begin{proof}
For $a \in [0,s)$, taking $j = \lfloor a n \rfloor$ in inequality \hyperref[ineq:lowerbd2]{(\ref*{ineq:lowerbd2})} above immediately implies 
\begin{align}\label{ineq:Hub}
U_s^\theta(x) &\leq H_{a, a,s,t}^\theta(x).
\end{align}
Fix $\delta \in (0,t)$; then for all $a \in [0, t - \delta)$ and all $u \in[0, a+\delta]$, we have
\begin{align}
Z_{1,1}(nu, n(a+\delta)) Z_{1, \lfloor ns \rfloor}(n (a + \delta), nt)\leq Z_{1 ,\lfloor ns \rfloor}(nu, nt). \label{ineq:intcoarse}
\end{align}
It then follows that
\begin{align*}
&P\left(\log \left(n \int_{a}^{a + \delta} \frac{Z_0^\theta(nu)}{Z_0^\theta(nt)}Z_{1,\lfloor n s \rfloor}(nu,nt)du \right) \geq nx \right) \\
&\hspace{3pc} \geq P\Big(\log Z_{1, \lfloor ns \rfloor}(n(a + \delta), nt) + \log \frac{Z_0^\theta(n(a+\delta))}{Z_0^\theta(nt)}\\
&\hspace{6pc}+ \log \left( n \int_{a}^{a + \delta} \frac{Z_0^\theta(nu)}{Z_0^\theta(n(a+\delta))}Z_{1,1}(nu,n(a+\delta) du\right) \geq nx \Big).
\end{align*}
Fix $\epsilon >0$. By independence of the Brownian environment, we find that
\begin{align}
&\frac{-1}{n} \log P\left(\log \left(n \int_{a}^{a + \delta} \frac{Z_0^\theta(nu)}{Z_0^\theta(nt)}Z_{1,\lfloor n s \rfloor}(nu,nt)du \right) \geq nx \right) \notag \\
&\hspace{3pc} \leq \frac{-1}{n}\log P\left(\log Z_{1, \lfloor ns \rfloor}(n(a + \delta), nt) + \log \frac{Z_0^\theta(n(a+\delta))}{Z_0^\theta(nt)} \geq n(x + \epsilon)\right)  \\
&\hspace{3pc} + \frac{-1}{n} \log P\left( \log \left( n \int_{a}^{a + \delta} \frac{Z_0^\theta(nu)}{Z_0^\theta(n(a+\delta))}Z_{1,1}(nu,n(a+\delta) du\right) \geq -n\epsilon \right). \label{varprob1error}
\end{align}
Applying the lower bound obtained by considering the minimum of the Brownian increments on the interval $[a, a+\delta]$ allows us to show that as $n \to \infty$ the probability in line \hyperref[varprob1error]{(\ref*{varprob1error})} tends to one. Then taking $\limsup$ and recalling inequality \hyperref[ineq:lowerbd1]{(\ref*{ineq:lowerbd1})}, we obtain
\begin{align}\label{ineq:Gub}
U_s^\theta(x) &\leq G_{a + \delta, s,t}^\theta(x + \epsilon).
\end{align}
By \hyperref[lem:Gcont]{Lemma \ref*{lem:Gcont}}, we may take $\delta,\epsilon \downarrow 0$ in \hyperref[ineq:Gub]{(\ref*{ineq:Gub})}. Optimizing over $a$ in the resulting equation and in \hyperref[ineq:Hub]{(\ref*{ineq:Hub})} gives the result.
\end{proof}
\begin{lemma}\label{lem:varprob2}
Fix $\theta > 0$, $(s,t) \in (0,\infty)^2$ and $x \in \mathbb{R}$. Then
\begin{align*}
U_s^\theta(x) &\geq \min \{ \inf_{0 \leq a < t} \left\{ G_{a,s,t}^\theta(x) \right\}, \inf_{0 \leq a < s} \left\{ H_{a,a,s,t}^\theta(x) \right\} \}.
\end{align*}
\end{lemma}
\begin{proof}
Fix a large $p > 1$ and small $\epsilon, \gamma > 0$. Consider uniform partitions $\{a_i\}_{i=0}^{M}$ of $[0,t]$ and $\{b_i\}_{i=0}^{N}$ of $[0,s]$ of mesh $\nu = \frac{t}{M+1}$ and $\delta = \frac{s}{N+1}$ respectively. We will add restrictions on these parameters later in the proof.  Take $n$ sufficiently large that $\lfloor b_i n \rfloor < \lfloor b_{i+1} n\rfloor $ for all $i$. 

Fix $j < \lfloor ns \rfloor$ not equal to any of the partition points $\lfloor b_i n \rfloor$ and consider $i$ so that $\lfloor b_i n \rfloor < j < \lfloor b_{i+1} n \rfloor$. Notice that $Z_0^\theta(nt)$ is $\sigma(B(nt))$ measurable and $ Z_j^\theta(0)$ is measurable with respect to $\sigma(B(s),B_1(s), \dots B_j(s) : s \leq 0)$, so these random variables and $Z_{j, \lfloor ns \rfloor}(u,v)$ are mutually independent if $0\leq u < v$. It follows from translation invariance and this independence that
\begin{align*}
&P\left(- \log Z_0^\theta(nt) + \log Z_j^\theta(0) + \log Z_{j, \lfloor ns \rfloor}(0,nt) \geq nx \right) \\
&\hskip80pt = P\left(- \log Z_0^\theta(nt) + \log Z_j^\theta(0) + \log Z_{j, \lfloor ns \rfloor}(n \gamma,n(t + \gamma)) \geq nx \right).
\end{align*}
We have
\begin{align*}
Z_{\lfloor b_i n \rfloor , \lfloor ns \rfloor}(0 , n(t + \gamma)) &\geq Z_{\lfloor b_i n \rfloor, j}(0 ,n \gamma) Z_{j, \lfloor ns \rfloor}(n \gamma, n(t + \gamma)).
\end{align*}
It then follows that
\begin{flalign*}
&P\left( - \log Z_0^\theta(nt) + \log Z_j^\theta(0) + \log Z_{j, \lfloor ns \rfloor}(0,nt) \geq nx \right) & \\
&\qquad \qquad\leq P\left( - \log Z_0^\theta(nt) + \log Z_{\lfloor b_{i+1} n \rfloor}^\theta(0) + \log Z_{\lfloor b_i n \rfloor , \lfloor ns \rfloor}(0 , n(t + \gamma)) \geq n (x - 2\epsilon ) \right) \\
&\qquad \qquad + P \left(\log Z_{\lfloor b_i n \rfloor, j}(0, n \gamma) \leq - n \epsilon \right) + P\left(\sum_{k= j+1}^{ \lfloor b_{i+1} n \rfloor} r_k^\theta(0) \leq - n \epsilon \right).
\end{flalign*}
Using the moment bound in  \hyperref[momentbound]{Lemma \ref*{momentbound}} with $\xi = -p$ for $p > 1$ and the exponential Markov inequality gives the bound
\begin{align*}
P \left( \log Z_{\lfloor b_i n \rfloor, j}(0, n\gamma) \leq - n \epsilon  \right) &\leq e^{- n p\left(\epsilon - p  \gamma - \delta \log\frac{\delta}{\gamma}\right) + o(n)} \leq e^{- n\frac{\epsilon}{2} p + o(n)}.
\end{align*}
For the last inequality, we first require $\gamma < \frac{\epsilon}{4p}$ and then take $\delta$ small enough that $\delta \log \frac{\delta}{\gamma} < \frac{\epsilon}{4}$.
The exponential Markov inequality and the known moment generating function of the i.i.d. sum give the bound
\begin{align*}
P\left(\sum_{k= j+1}^{ \lfloor b_{i+1} n \rfloor} r_k^\theta(0) \leq - n \epsilon \right) &\leq e^{-n p\left(\epsilon - \delta p^{-1}\log \left(\Gamma(\theta + p) \Gamma(\theta)^{-1}\right) \right)} \leq e^{- n p \frac{\epsilon}{2}}
\end{align*}
where in the last step we additionally require $\delta < \frac{\epsilon p}{4} \log \left(\Gamma(\theta + p) \Gamma(\theta)^{-1}\right)^{-1}$. For the case that $j$ is a partition point, we have
\begin{align*}
&P\left(- \log Z_0^\theta(nt) + \log Z_{\lfloor b_i n \rfloor}^\theta(0) + \log Z_{ \lfloor b_i n \rfloor, \lfloor ns \rfloor}(0,nt) \geq nx \right) \\
&\hspace{3pc} \leq P\left(- \log Z_0^\theta(nt) + \log Z_{\lfloor b_{i+1} n \rfloor}^\theta(0) +Z_{\lfloor b_i n \rfloor , \lfloor ns \rfloor}(0 , n(t + \gamma)) \geq n (x - 2\epsilon ) \right) \\
&\hspace{3pc} + P\left(\sum_{k= \lfloor b_i n \rfloor }^{ \lfloor b_{i+1} n \rfloor} r_k^\theta(0) \leq - 2n \epsilon \right).
\end{align*}
and the same error bound as above applies. We now turn to the problem of estimating the integral
\begin{align*}
&P \left( \log \left( n \int_{a_i}^{a_{i+1}} \frac{Z_0^\theta(nu)}{Z_0^\theta(nt)}Z_{1,\lfloor n s \rfloor}(nu,nt)du\right) \geq nx\right) \\
&\leq P \left( \log \left(\frac{Z_0^\theta(na_i)}{Z_0^\theta(nt)}Z_{1,\lfloor ns \rfloor}(na_i, nt)\right) \geq n (x - \epsilon) \right) \\
&\qquad \hspace{3pc} + P\left( \log \left( n \int_{a_i}^{a_{i+1}} \frac{Z_0^\theta(nu)}{Z_0^\theta(n a_i)}\frac{Z_{1,\lfloor n s \rfloor}(nu,nt)}{Z_{1,\lfloor n s \rfloor}(n a_i,nt)}du \right) \geq n \epsilon \right).
\end{align*}
By \hyperref[uppercoarse1]{Lemma \ref*{uppercoarse1}} we have
\begin{align*}
P\left( \log \left( n \int_{a_i}^{a_{i+1}} \frac{Z_0^\theta(nu)}{Z_0^\theta(n a_i)} \frac{Z_{1,\lfloor n s \rfloor}(nu,nt)}{Z_{1,\lfloor n s \rfloor}(n a_i,nt)}du \right) \geq n \epsilon \right) &\leq \exp\left\{-n  \left(\frac{\epsilon - \theta \nu}{2\sqrt{\nu}}\right)^2 +o(n)\right\}
\end{align*}
where we require $\nu < \frac{\epsilon}{\theta}$.

Take $n$ sufficiently large that $\log(ns + N) \leq n \epsilon$. It follows from \hyperref[ineq:upperbd1]{(\ref*{ineq:upperbd1})} and union bounds that
\begin{flalign*}
&\frac{1}{n} \log P \left( \sum_{k=1}^{\lfloor ns \rfloor} r_k^\theta(nt) \geq nx \right) \leq \frac{1}{n} \log(ns + N) & \\
&\qquad+\max_{0 \leq i \leq M-1}\Big\{\frac{1}{n} \log P \left( \log \left( n \int_{a_i}^{a_{i+1}} Z_0^\theta(nu) Z_0^\theta(nt)^{-1} Z_{1,\lfloor n s \rfloor}(nu,nt)du \right) \geq n(x- \epsilon) \right)\Big \}  \\
&\qquad \vee\max_{1 \leq j \leq \lfloor ns \rfloor}\Big\{ \frac{1}{n} \log P\left( -\log Z_0^\theta (nt) + \log Z_j^\theta(0) + \log Z_{j,\lfloor n s \rfloor}(0,nt)  \geq n(x - \epsilon) \right) \Big\}.
\end{flalign*}
Combining this with the previous estimates, multiplying by $-1$ and sending $n \to \infty$ gives
\begin{align*}
U_s^\theta(x) &\geq \min_{0 \leq i \leq M -1} \left\{G_{a_i s, t}^\theta(x - 2 \epsilon)\right\} \wedge \left(\frac{\epsilon - \theta \nu}{2\sqrt{\nu}}\right)^2 \wedge \frac{p \epsilon}{2} \wedge \min_{0 \leq i \leq N -1}\left\{ H_{b_{i+1}, b_i, s, t + \gamma}(x - 3 \epsilon)\right\}\\
&\geq \inf_{a \in [0,t)}\left\{G_{a,s,t}^\theta(x-2\epsilon)\right\}\wedge \left(\frac{\epsilon - \theta \nu}{2\sqrt{\nu}}\right)^2 \wedge \frac{p \epsilon}{2}  \\
&\wedge\inf_{a \in [0,s)}\left\{H_{a, a, s, t }(x) -  \sup_{a,b,b' \in [0,s] : |b-b'|< \delta}\left\{ |H_{a,b,s,t + \gamma}^\theta(x -3 \epsilon) - H_{a,b',s,t}^\theta(x)|\right\}\right\}.
\end{align*}
We first send $\delta \downarrow 0$, then $\gamma \downarrow 0$, then $\nu \downarrow 0$, then $p \uparrow \infty$. By \hyperref[lem:Ginf]{Lemma \ref*{lem:Ginf}}, there is $\eta > 0$ so that for all $\epsilon \in [0,1]$, we have
\begin{align*}
 \inf_{a \in [0,t )}\left\{G_{a,s,t}^\theta(x-2\epsilon)\right\} &= \inf_{a \in [0,t - \eta]}\left\{ G_{a,s,t}^\theta(x-2\epsilon)\right\}.
\end{align*} 
Now, take $\epsilon \downarrow 0$ and use Lemmas \hyperref[lem:Hcont]{\ref*{lem:Hcont}} and \hyperref[lem:Gcont]{\ref*{lem:Gcont}}. This gives the desired bound
\begin{align*}
U_s^\theta(x) &\geq \min\{\inf_{a \in [0,t)}\left\{G_{a,s,t}^\theta(x)\right\}, \inf_{a \in [0,t)} \left\{ H_{a,a,s,t}^\theta(x)\right\} \}. \qedhere
\end{align*}
\end{proof}

We now turn the variational problem for the right tail rate functions into a variational problem involving Legendre-Fenchel transforms.

\begin{lemma}\label{lem:lfvar}
For any $\theta > 0$ let $\xi \in (0,\theta)$. Then $J_{s,t}^*(\xi)$ satisfies the variational problem
\begin{align*}
0 = \max\bigg\{&\sup_{0 \leq a < t}\left\{  (t-a)\left(\frac{1}{2} \xi^2 - \theta \xi\right)  - s \log \frac{\Gamma(\theta - \xi )}{\Gamma(\theta)}  + J_{s,t-a}^*(\xi)\right\},  \\
&\sup_{0 \leq a < s}\left\{ t\left(\frac{1}{2} \xi^2 - \theta \xi\right)  - (s-a) \log \frac{\Gamma(\theta- \xi)}{\Gamma(\theta)} + (J_{s - a, t})^*(\xi)\right\} \bigg\}.
\end{align*}
\end{lemma}
\begin{proof}
\hyperref[lem:varprob1]{Lemma \ref*{lem:varprob1}} and \hyperref[lem:varprob2]{Lemma \ref*{lem:varprob2}} imply \hyperref[Uvarprob]{(\ref*{Uvarprob})}. Infimal convolution is Legendre-Fenchel dual to addition for proper convex functions \cite[Theorem 16.4]{Rock} so we find
\begin{align*}
(U_s^\theta)^*(\xi) &= \sup_{x \in \mathbb{R}} \left\{\xi x - \min\left\{\inf_{0 \leq a < t}\left\{ R_{t - a}^\theta \square J_{s,t - a}(x)\right\}, \inf_{0 \leq a < s}\left\{ R_t^\theta \square U_a^\theta \square J_{s - a, t}(x)\right\}\right\}\right\} \\
&= \sup_{x \in \mathbb{R}}\left\{ \max\{\sup_{0 \leq a < t}\left\{ \xi x -R_{t - a}^\theta \square J_{s,t - a}(x)\right\}, \sup_{0 \leq a < s}\left\{ \xi x - R_t^\theta \square U_a^\theta \square J_{s - a, t}(x)\right\}\right\} \\
&= \max\left\{\sup_{0 \leq a < t} \left\{ (R_{t-a}^\theta)^*(\xi) + J_{s,t-a}^*(\xi)\right\}, \sup_{0 \leq a < s}\left\{ (R_t^\theta)^*(\xi) + (U_a^\theta)^*(\xi) + (J_{s - a, t})^*(\xi) \right\}\right\}.
\end{align*}
If $\xi \in (0, \theta)$, then $(U_s^\theta)^*(\xi) < \infty$, so we may subtract $(U_s^\theta)^*(\xi)$ from both sides. Substituting in the known Legendre-Fenchel transforms gives the result.
\end{proof}

\subsection{Solving the variational problem}
Next, we show that the variational problem in \hyperref[lem:lfvar]{Lemma \ref*{lem:lfvar}} identifies $J_{s,t}^*(\xi)$ for $\xi > 0$. To show the analogous result in \cite{GS13}, the authors followed the approach of rephrasing the variational problem as a Legendre-Fenchel transform in the space-time variables and appealing to convex analysis. We present an alternate method for computing $J_{s,t}^*(\xi)$ in the next proposition, which has the advantage of allowing us to avoid some of the technicalities in that argument. This direct approach is the main reason we are able to appeal to the G\"artner-Ellis theorem to prove the large deviation principle.
\begin{proposition}\label{prop:varsol}
Let $I \subseteq \mathbb{R}$ be open and connected and let $h,g:I \to \mathbb{R}$ be twice continuously differentiable functions with $h'(\theta) > 0$ and $g'(\theta) < 0$ for all $\theta \in I$. For $(x,y) \in (0,\infty)^2$, define
\begin{align*}
f_{x,y}(\theta) = x h(\theta) + y g(\theta)
\end{align*}
and suppose that $\frac{d^2}{d\theta^2}f_{x,y}(\theta) > 0$ for all $(x,y) \in (0,\infty)^2$ and that $f_{x,y}(\theta) \to \infty$ as $\theta \to \partial I$ (which may be a limit as $\theta \to \pm \infty$). If $\Lambda(x,y)$ is a continuous function on $(0,\infty)^2$ with the property that for all $(x,y) \in (0,\infty)^2$ and $\theta \in I$ the identity
\begin{align}\label{genvarprob}
0&=  \sup_{0 \leq a < x}\left\{ \Lambda(x-a, y)  - f_{x-a,y}(\theta) \right\}\vee \sup_{0 \leq b < y}\left\{ \Lambda(x, y - b) - f_{x,y-b}(\theta) \right\}
\end{align} 
holds, then
\begin{align*}
\Lambda(x,y) &= \min_{\theta \in I}\left\{ f_{x,y}(\theta)\right\}.
\end{align*}
\end{proposition}

\begin{proof}
Fix $(x,y) \in (0,\infty)^2$ and call  $\nu = \frac{y}{x}$. Under these hypotheses, there exists a unique $\theta_{x,y}^*  = \arg \min_{\theta \in I} f_{x,y}(\theta) = \theta_{1,\nu}^ *$. Identity \hyperref[genvarprob]{(\ref*{genvarprob})} implies that for all $a \in [0,x)$ and $b \in [0,y)$ we have 
\begin{align*}
\Lambda(x-a,y) \leq f_{x-a,y}(\theta_{x-a,y}^*), \qquad \Lambda(x,y-b) \leq f_{x,y-b}(\theta_{x,y-b}^*),
\end{align*}
and therefore for any $\theta \in I$, $a \in [0,x)$ and $b \in [0,y)$,
\begin{align}
\Lambda(x-a,y) - f_{x-a,y}(\theta) &\leq f_{x-a,y}(\theta_{x-a,y}^*) - f_{x-a,y}(\theta), \label{varub1} \\\Lambda(x,y-b) - f_{x,y-b}(\theta) &\leq f_{x,y-b}(\theta_{x,y-b}^*) - f_{x, y-b}(\theta).\label{varub2}
\end{align}
Uniqueness of minimizers implies that $f_{x-a,y}(\theta_{x-a,y}^*) - f_{x-a,y}(\theta) < 0$ unless $\theta = \theta_{x-a,y}^*$ and similarly $f_{x,y-b}(\theta_{x,y-b}^*) - f_{x, y-b}(\theta) < 0$ unless $\theta = \theta_{x,y-b}^*$. Notice that $\theta_{1,\nu}^*$ solves
\begin{align}\label{eqFOC}
0 &= h'(\theta_{1,\nu}^*) + \nu g'(\theta_{1,\nu}^*).
\end{align}
By the implicit function theorem, we may differentiate the previous expression with respect to $\nu$ to obtain
\begin{align}\label{optinc}
\frac{d \theta_{1,\nu}^*}{d\nu} &= - \frac{g'(\theta_{1,\nu}^*)}{h''(\theta_{1,\nu}^*) + \nu g''(\theta_{1,\nu}^*)} > 0.
\end{align}

Now, set $\theta = \theta_{x,y}^*$ in \hyperref[genvarprob]{(\ref*{genvarprob})}. Equality \hyperref[{optinc}]{(\ref*{optinc})} implies that for $a \in(0,x)$ and $b \in (0,y)$, $\theta_{(x,y-b)}^*  < \theta_{(x,y)}^* < \theta_{(x-a,y)}^*$. Then  \hyperref[varub1]{(\ref*{varub1})} and \hyperref[varub2]{(\ref*{varub2})} give us the inequalities
\begin{align}
\Lambda(x-a,y) - f_{x-a,y}(\theta_{x,y}^*) &\leq f_{x-a,y}(\theta_{x-a,y}^*) - f_{x-a,y}(\theta_{x,y}^*) < 0, \label{varub3} \\
\Lambda(x,y-b) - f_{x,y-b}(\theta_{x,y}^*) &\leq f_{x,y-b}(\theta_{x,y-b}^*) - f_{x, y-b}(\theta_{x,y}^*) < 0. \label{varub4}
\end{align}
Notice that \hyperref[genvarprob]{(\ref*{genvarprob})} implies either there exists $a_n \to a \in [0,x]$ or $b_n \to b \in [0,y]$ so that one of the following hold: 
\begin{align*}
\Lambda(x-a_n,y) - f_{x-a_n,y}(\theta_{x,y}^*) \to 0, \qquad \Lambda(x,y-b_n) - f_{x,y-b_n}(\theta_{x,y}^*) \to 0.
\end{align*}
Our goal is to show that the only possible limits are $a_n \to 0$ or $b_n \to 0$, from which the result follows from continuity. Continuity and inequalities \hyperref[varub3]{(\ref*{varub3})} and \hyperref[varub4]{(\ref*{varub4})} rule out the possibilities $a \in (0,x)$ and $b \in (0,y)$ respectively. It therefore suffices to show that
\begin{align}
&\limsup_{a \to x^-} f_{x-a,y}(\theta_{x-a,y}^*) - f_{x-a,y}(\theta_{x,y}^*) <0, \label{suffbd1} \\
&\limsup_{b \to y^-} f_{x,y-b}(\theta_{x,y-b}^*) - f_{x, y-b}(\theta_{x,y}^*) < 0. \label{suffbd2}
\end{align}
We will only write out the proof of \hyperref[suffbd1]{(\ref*{suffbd1})}, since the proof of \hyperref[suffbd2]{(\ref*{suffbd2})} is similar. For any fixed $a \in (0,x)$, we have
\begin{align*}
f_{x-a,y}(\theta_{x-a,y}^*) - f_{x-a,y}(\theta_{x,y}^*) < 0.
\end{align*}
It suffices to show that the previous expression is decreasing in $a$.  Differentiating the previous expression and using \hyperref[eqFOC]{(\ref*{eqFOC})} and the fact that $\theta_{(x,y)}^* < \theta_{(x-a,y)}^*$, we find
\begin{align*}
&\frac{d}{da}\left( (x-a) h(\theta_{(x-a,y)}^*) + y g(\theta_{(x-a,y)}^*) - \left[(x-a) h(\theta_{(x,y)}^*) + y g(\theta_{(x,y)}^*)\right]\right) \\
 &\hskip50pt = h(\theta_{(x,y)}^*) - h(\theta_{(x-a,y)}^*) < 0. \qedhere
\end{align*}
\end{proof}
\begin{corollary}\label{cor:lfform}
For all $\xi > 0$,
\begin{align*}
J_{s,t}^*(\xi) &= \min_{\theta > \xi}\left\{ t\left(- \frac{\xi^2}{2} + \theta \xi\right) + s \log \frac{\Gamma(\theta - \xi)}{\Gamma(\theta)}\right\} \\
&= \min_{\mu > 0}\left\{ t\left(\frac{\xi^2}{2} + \xi \mu \right) - s \log \frac{\Gamma(\mu + \xi)}{\Gamma(\mu)} \right\}.
\end{align*}
\end{corollary}
\begin{proof}
It follows from the variational representation in \hyperref[lem:lfvar]{Lemma \ref*{lem:lfvar}} that $J_{s,t}^*(\xi)$ is not infinite for any choice of the parameters $\xi,s,t>0$. It then follows from \hyperref[lem:lfreg]{Lemma \ref*{lem:lfreg}} and \cite[Theorem 10.1]{Rock} that $J_{s,t}^*(\xi)$ is continuous in $(s,t) \in (0,\infty)^2$.

Fix $\xi$ and set $I = \{ \theta : \theta > \xi\}$. For $\theta \in I$ and $s,t \in(0,\infty)$, define 
\begin{align*}
&h(\theta) = - \frac{\xi^2}{2} + \theta \xi,  &g(\theta) = \log \frac{\Gamma(\theta - \xi)}{\Gamma(\theta)}, \\ 
&f_{s,t}(\theta) = sh(\theta) + tg(\theta), &\Lambda(s,t) = J_{s,t}^*(\xi).
\end{align*}
\hyperref[lem:lfvar]{Lemma \ref*{lem:lfvar}} shows that with these definitions $J_{s,t}^*(\xi)$ solves the variational problem in  \hyperref[prop:varsol]{Proposition \ref*{prop:varsol}}. Because $\Psi_1(x) > 0$ and $\Psi_2(x) < 0$, we see that for $\theta \in I$
\begin{align*}
&g'(\theta) = \Psi_0(\theta - \xi) - \Psi_0(\theta) < 0, &g''(\theta) = \Psi_1(\theta - \xi) - \Psi_1(\theta) > 0.
\end{align*}
It then follows that $\frac{d^2}{d\theta^2}f_{s,t}(\theta) > 0$. Moreover, since $\log \frac{\Gamma(\theta - \xi)}{\Gamma(\theta)}$ grows like $-\xi \log(\theta)$ at infinity and $-\log(\theta - \xi)$ at $\xi$, $f_{s,t}(\theta)$ also tends to infinity at the boundary of $I$ and the result follows.

The second equality is the substitution $\mu = \theta - \xi$.
\end{proof}

\section{Moment Lyapunov exponents and the LDP}
The next result would be Varadhan's theorem if $J_{s,t}(x)$ were a full rate function, rather than a right tail rate function. The proof is somewhat long and essentially the same as the proof of Varadhan's theorem, so we omit it. Details of a similar argument for the stationary log-gamma polymer can be found in \cite[Lemma 5.1]{GS13}. The exponential moment bound needed for the proof follows from \hyperref[momentbound]{Lemma \ref*{momentbound}}.
\begin{lemma}\label{lem:MLEpos}
For $\xi > 0$,
\begin{align*}
J_{s,t}^*(\xi) &= \lim_{n \to \infty} \frac{1}{n} \log E \left[e^{\xi \log Z_{1, \lfloor ns \rfloor}(0,nt)} \right]
\end{align*}
and in particular the limit exists.
\end{lemma}

\begin{remark}
\hyperref[lem:MLEpos]{Lemma \ref*{lem:MLEpos}} shows that $J_{s,t}^*(\xi)$ is the $\xi$ moment Lyapunov exponent for the parabolic Anderson model associated to this polymer. With this in mind, the second formula in the statement of \hyperref[cor:lfform]{Corollary \ref*{cor:lfform}} above agrees with the conjecture in \cite[Appendix A.1]{BC14b}.

To see this, we first observe that the partition function we study differs slightly from the partition function $Z_\beta(t,n)$ studied in \cite{BC14b} (defined in equation (3) of that paper). As we saw was the case for $Z_n(\beta)$ in equation \hyperref[condexp]{(\ref*{condexp})}, up to normalization constants both $Z_{0,\lfloor ns \rfloor}(0,nt)$ and $Z_\beta(t,n)$ are conditional expectations of functionals of a Poisson path. The normalization constant for $Z_{0, \lfloor ns \rfloor}(0,nt)$ is given by the Lebesgue measure of the Weyl chamber $A_{\lfloor ns \rfloor+1, nt}$, while the normalization constant for $Z_\beta(t,n)$ is $P_{\pi(0)=0}\left(\pi(t) = n\right)$ where $\pi(\cdot)$ is again a rate one Poisson process. There is a further difference in that \cite{BC14b} adds a pinning potential of strength $\frac{\beta}{2}$ at the origin to the definition of $Z_\beta(t,n)$, which introduces a multiplicative factor of $e^{-\frac{\beta}{2}t}$. Combining these changes and restricting to the parameters studied in \cite[Appendix A.1]{BC14b}, we have the relation
\begin{align*}
e^{- \frac{n}{2}}\frac{P_{\pi(0) = 0}\left(\pi(n) = \lfloor n \nu \rfloor\right)}{|A_{\lfloor n \nu \rfloor + 1,n}|}Z_{0, \lfloor n \nu \rfloor}(0,n) &= Z_1(n, \lfloor n \nu\rfloor).
\end{align*}
Since $P_{\pi(0) = 0}\left(\pi(n) = \lfloor n \nu \rfloor\right)|A_{\lfloor n \nu \rfloor+1,n}|^{-1} = e^{-n}$, \hyperref[cor:lfform]{Corollary \ref*{cor:lfform}} and \hyperref[lem:MLEpos]{Lemma \ref*{lem:MLEpos}} then imply that for any $k > 0$,
\begin{align*}
\lim_{n \to \infty} \frac{1}{n} \log E\left[ Z_1(n, \lfloor n \nu \rfloor)^k \right] &= - \frac{3}{2}k + \min_{z > 0}\left\{ \frac{k^2}{2} + k z  - \nu \log \frac{\Gamma(z + k)}{\Gamma(z)} \right\} \\
&= \min_{z > 0}\left\{\frac{k(k-3)}{2} + k z- \nu \log \frac{\Gamma(z + k)}{\Gamma(z)} \right\},
\end{align*}
which is the extension of the moment Lyapunov exponent $H_k(z_k^0)$ conjectured in the middle of page 24 of \cite{BC14b}.
\end{remark}

Our next goal is to show that the left tail large deviations are not relevant at the scale we consider. This proof is based on the proof of \cite[Lemma 4.2]{GS13} which contains a small mistake; as currently phrased, the argument in that paper only works for $s,t \in \mathbb{Q}$. This problem can be fixed by altering the geometry of the proof, but doing this adds some technicalities which can be avoided in the model we study. We will follow an argument similar to the original proof for $s \in \mathbb{Q}$, then show that this implies what we need for all $s$.
\begin{proposition} \label{prop:lowerbd}
Fix $s,t >0$. For all $\epsilon > 0$
\begin{align*}
\liminf_{n \to \infty} - \frac{1}{n} \log P\left(\log Z_{1, \lfloor ns \rfloor}(0,nt) \leq n(\rho(s,t) - \epsilon) \right) & = \infty.
\end{align*}
\end{proposition}
\begin{proof}

First we consider the case $s \in \mathbb{Q}$. There exists $M \in \mathbb{N}$ large enough that $M(s \wedge t) \geq 1$ and for all $m \geq M$ we have
\begin{align*}
\frac{1}{m} E \log Z_{1, \lfloor ms \rfloor}(0,mt) \geq \rho(s,t) - \epsilon.
\end{align*}
Fix $m \geq M$ so that $ms \in \mathbb{N}$. We will denote coordinates in $\mathbb{R}^{\lfloor ns \rfloor - 1}$ by $(u_1, \dots , u_{\lfloor ns \rfloor -1})$. For $a,b,s,t \in (0,\infty)$ and $n,k,l \in \mathbb{N}$, define a family of sets $A_{k,a}^{l,b} \subset \mathbb{R}^{\lfloor ns \rfloor - 1}$ by
\begin{align*}
A_{k,a}^{l,b} = &\{0 < u_1 < \dots < u_{k-1} < a < u_k < \dots < u_{k + l -1} < a + b < u_{k+l} < \dots < u_{\lfloor ns \rfloor - 1} < nt\}.
\end{align*}
For $j,k, \in \mathbb{Z}^+$, set
\begin{align*}
A_j^k &\equiv A_{j \lfloor ms \rfloor + 1, (j + k) mt}^{\lfloor ms \rfloor, mt}.
\end{align*}
For each $n$ sufficiently large that the expression below is greater than one, define
\begin{align*}
N &= \left \lfloor \frac{n}{m} - \lfloor \sqrt{n} \rfloor - 2 \right \rfloor,
\end{align*}
so that we have
\begin{align}
(n - 2m)t &\leq (N + \lfloor\sqrt{n}\rfloor + 1)mt \leq (n-m)t, \label{Ntbound}\\
\left(\lfloor \sqrt{n} \rfloor + 1\right)ms - 1&\leq \lfloor ns \rfloor - N \lfloor ms \rfloor \leq \left(\lfloor\sqrt{n}\rfloor + 2\right)ms. \label{Nsbound}
\end{align}
With this choice of $N$, for $0 \leq k \leq \lfloor \sqrt{n} \rfloor$ and $0 \leq j \leq N-1$, $A_j^k$ is nonempty. Then for $0 \leq k \leq \lfloor \sqrt{n} \rfloor$, define
\begin{align*}
D_k &= \cap_{j=0}^{N-1} A_j^k \cap \left\{u : 0 < u_1 < \dots < u_{(N+1)ms - 1} <  t\left(n - \frac{m}{2}\right) < u_{(N+1)ms} < \dots < u_{\lfloor ns \rfloor - 1} < nt \right\}.
\end{align*}
To simplify the formulas that follow, we introduce the notation $s_j = j ms$ and $t_j^k = (j+k)mt$.  In words, we can think of $D_k$ as the collection of paths from $0$ to $nt$ which traverse the bottom line until $t_0^k$, then for $0 \leq j \leq N-1$ move from $t_j^k$ to $t_{j+1}^k$ along the next $ms$ lines. The path then moves from $t_N^k$ to $t\left(n - \frac{m}{2}\right)$ along the next $ms$ lines and finally proceeds to $nt$ along the remaining lines. Observe that $\{D_k\}_{k=0}^{\lfloor \sqrt{n} \rfloor}$ is a pairwise disjoint, non-empty family of sets. With the convention $u_0 = 0$ and $u_{\lfloor ns \rfloor} = nt$, we have the bound
\begin{align*}
Z_{1, \lfloor ns \rfloor}(0,nt) &\geq \sum_{k=0}^{\lfloor \sqrt{n}\rfloor} \int_{D_k} e^{\sum_{i=1}^{\lfloor ns \rfloor} B_i(u_{i-1}, u_{i})} du_1 \dots u_{\lfloor ns \rfloor -1}.
\end{align*}
In the integral over $D_k$, for each $0 \leq j \leq N$ we add and subtract $B_{s_j}(t_j^k)$ in the exponent. Similarly, add and subtract $B_{s_{N+1}}\left(t\left(n - \frac{m}{2}\right)\right)$. The reason for this step is that this will make the product of integrals coming from the definition of $D_k$ into a product of partition functions, as when we showed supermultiplicativity of the partition function in \hyperref[submult]{(\ref*{submult})}. Introduce 
\begin{align*}
H_k^n &= \inf_{t_N^k = u_0 < u_1 < \dots < u_{ms - 1} < u_{ms} =  n\left(t - \frac{m}{2}\right)}\left\{ \sum_{i=0}^{ms-1}B_{s_N + i}(u_{i-1}, u_i)\right\}
\end{align*}
and observe that $H_0^n \leq B_{s_N}(t_N^0, t_N^k) + H_k^n$. Let $C> 0$ be a uniform lower bound in $n$ (recall that $m$ is fixed) on the Lebesgue measure of the Weyl chamber in the definition of $H_{\lfloor \sqrt{n} \rfloor}^n$. Such a bound exists by \hyperref[Ntbound]{(\ref*{Ntbound})}. Set $I_n = \max_{t_{N-1}^0 \leq u \leq t_{N-1}^{\lfloor \sqrt{n}\rfloor}}\{B_{s_N}(t_{N-1}^0,  u)\}$. We have the lower bound
\begin{align*}
Z_{1, \lfloor ns \rfloor}(0,nt) \geq C Z_{s_{N+1}, \lfloor ns \rfloor}\left(t\left(n - \frac{m}{2}\right), nt\right)e^{H_0^n - I_n}\left(\sum_{k=0}^{\lfloor \sqrt{n} \rfloor} e^{B_0(0, t_0^k)} \prod_{j=0}^{N-1}Z_{s_j,s_{j+1}}\left(t_j^k, t_{j+1}^k\right) \right).
\end{align*}
We therefore have the upper bound
\begin{align*}
&P\left(\log Z_{1, \lfloor ns \rfloor}(0,nt) \leq -n(\rho(s,t) - 6\epsilon) \right) \\
&\hspace{3pc} \leq P\left(\log Z_{(N+1)ms, \lfloor ns \rfloor}\left(t\left(n - \frac{m}{2}\right), nt\right) \leq -n \epsilon - \log C \right) \\
&\hspace{3pc}+P\left(\max_{0 \leq k \leq \lfloor \sqrt{n} \rfloor} \sum_{j=0}^{N-1}\log Z_{s_j + 1,s_{j+1}}\left(t_j^k, t_{j+1}^k\right) \leq -n(\rho(s,t) - 2\epsilon) \right) \\
&\hspace{3pc}+ P\left(H_0 \leq -n \epsilon \right) + P\left(\min_k B_0(t_0^k) \leq -n\epsilon \right) + P\left(I_n \geq n \epsilon \right).
\end{align*}
It follows from translation invariance, \hyperref[lem:normtail]{Lemma \ref*{lem:normtail}}, and \hyperref[Nsbound]{(\ref*{Nsbound})} that \begin{align*}
P\left(\log Z_{(N+1)ms, \lfloor ns \rfloor}\left(t\left(n - \frac{m}{2}\right), nt\right) \leq -n \epsilon - \log \frac{mt}{12} \right) = O\left(e^{-n^\frac{3}{2}}\right).
\end{align*} 
We have
\begin{align*}
&P\left(\max_{0 \leq k \leq \lfloor \sqrt{n} \rfloor} \left\{ \sum_{j=0}^{N-1}\log Z_{s_j + 1,s_{j+1}}\left(t_j^k, t_{j+1}^k\right) \right\} \leq -n(\rho(s,t) - 2\epsilon) \right) \\
&\hspace{3pc} = P\left(\sum_{j=0}^{N-1}\log Z_{s_j + 1,s_{j+1}}\left(t_j^1, t_{j+1}^1\right) \leq -n(\rho(s,t) - 2\epsilon) \right)^{\lfloor \sqrt{n}\rfloor} = O\left(e^{- c_1 n^{\frac{3}{2}}}\right)
\end{align*}
for some $c_1 > 0$. The first equality comes from the fact that the terms in the maximum are i.i.d. and the second comes from large deviation estimates for an i.i.d. sum once we recall that $N = \frac{n}{m} + o(n)$. 

Recall that by \hyperref[Ntbound]{(\ref*{Ntbound})}, $n\left(t - \frac{m}{2}\right) - t_N^0 = O(\sqrt{n})$. It follows from \hyperref[lem:HGUE]{Lemma \ref*{lem:HGUE}} that there exist $c_2, C_2 > 0$ so that
\begin{align*}
P\left(H_0 \leq -n \epsilon \right) &\leq C_2 e^{- c_2 n^{\frac{3}{2}}}.
\end{align*}
The remaining two terms can be controlled with the reflection principle and are $O\left(e^{-\frac{1}{2}n^{\frac{3}{2}}}\right)$.

Now let $s$ be irrational. For each $k$, fix $\tilde{s}_k < s$ rational with $e^{-k} < |\tilde{s}_k - s| < 2 e^{-k}$ and set $\tilde{t}_k = t -\frac{1}{k} $. Call $\alpha_k = s - \tilde{s}_k$ and $\beta_k = t - \tilde{t}_k = \frac{1}{k}$. Subadditivity gives
\begin{align*}
P\left(\log Z_{1, \lfloor ns \rfloor}(0,nt) \leq n(\rho(s,t) - \epsilon) \right) &\leq P\left(\log Z_{1, \lfloor n \tilde{s}_k \rfloor}(0,n\tilde{t}_k) \leq n\left(\rho(\tilde{s}_k,\tilde{t}_k) - \frac{\epsilon}{2}\right)  \right) \\
&+ P\left( \log Z_{\lfloor n \tilde{s}_k\rfloor, \lfloor ns \rfloor }(n \tilde{t}_k, nt) \leq n\left(\rho(s,t) - \rho(\tilde{s}_k, \tilde{t}_k) - \frac{\epsilon}{2}\right) \right).
\end{align*}

Since $\tilde{s}_k$ is rational, we have already shown that the first term is negligible. Take $k$ sufficiently large that $\rho(s,t) - \rho(\tilde{s}_k, \tilde{t}_k) - \frac{\epsilon}{2} < -\frac{\epsilon}{4}$. By \hyperref[lem:JGUE]{Lemma \ref*{lem:JGUE}}, we find
\begin{align*}
&\liminf_{n \to \infty} - \frac{1}{n} P\left(\log Z_{1, \lfloor ns \rfloor}(0,nt) \leq n(\rho(s,t) - \epsilon) \right) \geq \alpha_k J_{GUE}\left(\frac{\frac{\epsilon}{4} - \alpha_k \log \beta_k - \alpha_k + \alpha_k \log \alpha_k}{2\sqrt{\alpha_k \beta_k}} \right).
\end{align*}
Using formula \hyperref[JGUE]{(\ref*{JGUE})}, $J_{GUE}(r) = 4 \int_0^r \sqrt{x(x+2)}dx$, it is not hard to see that as $k \to \infty$, this lower bound tends to infinity.
\end{proof}

\begin{lemma}\label{lem:MLEneg}
Fix $s,t >0$ and $\xi < 0$. Then
\begin{align*}
\lim_{n \to \infty} \frac{1}{n} \log E\left[e^{\xi \log Z_{1, \lfloor ns \rfloor}(0,nt)}\right] &= \xi \rho(s,t).
\end{align*}
\end{lemma}
\begin{proof}
Fix $\epsilon > 0$ and small and recall that \hyperref[momentbound]{Lemma \ref*{momentbound}} and Jensen's inequality imply that for any $\xi < 0$, $\sup_n\left\{ \frac{1}{n} \log E\left[e^{\xi \log Z_{1, \lfloor ns \rfloor}(0,nt)}\right] \right\}  < \infty$. The lower bound is immediate from
\begin{align*}
\frac{1}{n} \log E\left[e^{\xi \log Z_{1, \lfloor ns \rfloor}(0,nt)}\right] &\geq \frac{1}{n} \log E\left[e^{\xi \log Z_{1, \lfloor ns \rfloor}(0,nt)}1_{\{\log Z_{1,\lfloor ns \rfloor}(0,nt) \leq n(\rho(s,t) + \epsilon)\}}\right] \\
&\geq \xi (\rho(s,t) + \epsilon) + \frac{1}{n} \log P(\log Z_{1,\lfloor ns \rfloor}(0,nt) \leq n(\rho(s,t) + \epsilon))
\end{align*}
once we recall that $P(\log Z_{1,\lfloor ns \rfloor}(0,nt) \leq n(\rho(s,t) + \epsilon)) \to 1$.

For the upper bound, we decompose the expectation as follows
\begin{align*}
E\left[e^{\xi \log Z_{1, \lfloor ns \rfloor}(0,nt)}\right] =&E\left[e^{\xi \log Z_{1, \lfloor ns \rfloor}(0,nt)}1_{\{\log Z_{1,\lfloor ns \rfloor}(0,nt) > n(\rho(s,t) - \epsilon)\}}\right]\\
&+ E\left[e^{\xi \log Z_{1, \lfloor ns \rfloor}(0,nt)}1_{\{\log Z_{1,\lfloor ns \rfloor}(0,nt) \leq n(\rho(s,t) - \epsilon)\}}\right].
\end{align*}
Recalling that $P\left( \log Z_{1,\lfloor ns \rfloor}(0,nt) > n(\rho(s,t) - \epsilon)\right) \to 1$, this leads to
\begin{align*}
&\limsup_{n \to \infty}\frac{1}{n} \log E\left[e^{\xi \log Z_{1, \lfloor ns \rfloor}(0,nt)}\right]  \\
&\hskip30pt \leq \max \Big \{ \xi (\rho(s,t) - \epsilon), \limsup_{n \to \infty} \frac{1}{n} \log E\left[e^{\xi \log Z_{1, \lfloor ns \rfloor}(0,nt)}1_{\{\log Z_{1,\lfloor ns \rfloor}(0,nt) \leq n(\rho(s,t) - \epsilon)\}}\right]  \Big \}.
\end{align*}
By Cauchy-Schwarz and \hyperref[prop:lowerbd]{Proposition \ref*{prop:lowerbd}}
\begin{align*}
&\limsup_{n \to \infty}\frac{1}{n} \log E\left[e^{\xi \log Z_{1, \lfloor ns \rfloor}(0,nt)}1_{\{\log Z_{1,\lfloor ns \rfloor}(0,nt) \leq n(\rho(s,t) - \epsilon)\}}\right]  \\
&\hskip30pt \leq \frac{1}{2} \sup_n \left\{ \frac{1}{n} \log E\left[e^{2 \xi \log Z_{1, \lfloor ns \rfloor}(0,nt)}\right] \right\} \\
&\hskip30pt + \limsup_{n \to \infty}\frac{1}{2n} \log P\left( \log Z_{1,\lfloor ns \rfloor}(0,nt) \leq n(\rho(s,t) - \epsilon)\right) = -\infty. \qedhere
\end{align*}
\end{proof}

Combining the previous results, we are led to the proof of \hyperref[thm:MLE]{Theorem \ref*{thm:MLE}}, from which we immediately deduce \hyperref[thm:LDP]{Theorem \ref*{thm:LDP}}.

\begin{proof}[Proof of Theorem \ref*{thm:MLE}]
Lemmas \hyperref[lem:MLEpos]{\ref*{lem:MLEpos}} and \hyperref[lem:MLEneg]{\ref*{lem:MLEneg}} give the limit for $\xi \neq 0$ and the limit for $\xi = 0$ is zero.

Note that $\Lambda_{s,t}(\xi)$ is differentiable for $\xi < 0$ the left derivative at zero is $\rho(s,t)$. For $\xi > 0$, there is a unique $\mu(\xi)$ solving
\begin{align}\label{Lambdaeq}
\Lambda_{s,t}(\xi) &= t\left(\frac{\xi^2}{2} + \xi \mu(\xi)\right) - s \log \frac{\Gamma(\mu(\xi) + \xi)}{\Gamma(\mu(\xi))}.
\end{align}
This $\mu(\xi)$ is given by the unique solution to
\begin{align}\label{musolves}
0 &= t \xi + s\left(\Psi_0(\mu(\xi)) - \Psi_0(\mu(\xi) + \xi)\right),
\end{align}
which can be rewritten as
\begin{align*}
\frac{1}{\xi}\left( \Psi_0(\mu(\xi) + \xi) - \Psi_0(\mu(\xi)) \right) &= \frac{t}{s}.
\end{align*}
By the mean value theorem, there exists $x \in [0,\xi)$ so that
\begin{align*}
\Psi_1^{-1}\left(\frac{t}{s}\right) - x &= \mu(\xi).
\end{align*}
Using this, we see that $\Lambda_{s,t}(\xi)$ is continuous at $\xi = 0$. The implicit function theorem implies that $\mu(\xi)$ is smooth for $\xi > 0$. Differentiating \hyperref[Lambdaeq]{(\ref*{Lambdaeq})} with respect to $\xi$ and applying \hyperref[musolves]{(\ref*{musolves})}, we have
\begin{align*}
\frac{d}{d\xi} \Lambda_{s,t}(\xi)&= t\left(\xi + \mu(\xi) \right) - s \Psi_0(\mu(\xi) + \xi).
\end{align*}
Substituting in for $\mu(\xi)$, appealing to continuity, and taking $\xi \downarrow 0$ gives
\begin{align*}
\lim_{\xi \downarrow 0} \frac{d}{d\xi} \Lambda_{s,t}(\xi) &= t \Psi_1^{-1}\left(\frac{t}{s}\right) - s \Psi_0\left(\Psi_1^{-1} \left( \frac{t}{s}\right) \right) \\
&= \rho(s,t)
\end{align*}
which implies differentiability at zero and hence at all $\xi$. 
\end{proof}
\begin{proof}[Proof of Theorem \ref*{thm:LDP}]
The large deviation principle holds by \hyperref[thm:MLE]{Theorem \ref*{thm:MLE}} and the G\"artner-Ellis theorem \cite[Theorem 2.3.6]{DemboZeitouni}.
\end{proof}

\begin{appendices}
\section{Right tail rate functions} \label{sec:existreg}
\subsection{Existence and structure of the right tail rate function} \label{subsec:existreg}
We now turn to the problem of showing the existence and regularity of the right tail rate function for the polymer free energy. Our main goal in this subsection will be to prove \hyperref[thm:exist]{Theorem \ref*{thm:exist}}. As is typical for right tail large deviations, existence and regularity follow from (almost) subadditivity arguments. Because the partition function degenerates for steps with no time component and we do not restrict attention to integer $s$, it is necessary to tilt time slightly in this argument. We break the proof of \hyperref[thm:exist]{Theorem \ref*{thm:exist}} into two parts: first we show the result with time tilted, then we show that this change does not matter.
\begin{theorem}\label{thm:existreg}
For all $s\geq 0$, $t > 0$ and $r \in \mathbb{R}$, the limit
\begin{align*}
J_{s,t}(r) &= \lim_{x \to \infty} - \frac{1}{x} \log P \left(\log Z_{0, \lfloor xs \rfloor}(0, xt - 1) \geq xr \right)
\end{align*}
exists and is $\mathbb{R}_+$ valued.  Moreover, $J_{s,t}(r)$ is continuous, convex, subadditive, and positively homogeneous of degree one as a function of $(s, t,r) \in [0,\infty) \times(0,\infty) \times\mathbb{R}$. For fixed $s$ and $t$, $J_{s,t}(r)$ is increasing in $r$ and $J_{s,t}(r) = 0$ if $r \leq \rho(s,t)$.
\end{theorem}

\begin{proof}
Define the function $T: [0,\infty) \times (1,\infty) \times \mathbb{R} \to \mathbb{R}_+$ by
\begin{align*}
T(x,y,z)&= - \log P \left( \log Z_{0, \lfloor x \rfloor}(0,y - 1) \geq z \right).
\end{align*}
\hyperref[lem:existlb]{Lemma \ref*{lem:existlb}} in the appendix implies that $P \left( \log Z_{0, \lfloor x \rfloor}(0,y - 1) \geq z \right) \neq 0$ and therefore that this function is well-defined.

Take $(x_1, y_1, z_1),(x_2, y_2, z_2) \in [0,\infty) \times(1,\infty) \times \mathbb{R}$ and call $x_{1,2} = \lfloor x_1 + x_2 \rfloor - \lfloor x_1 \rfloor - \lfloor x_2 \rfloor \in \{0,1\}$. By \hyperref[gensubmult]{(\ref*{gensubmult})}, we have
\begin{align*}
Z_{0, \lfloor x_1 + x_2 \rfloor}(0,y_1 + y_2 - 1)&\geq  Z_{0, \lfloor x_1 \rfloor}(0,y_1 - 1) Z_{\lfloor x_1 \rfloor, \lfloor x_1 + x_2 \rfloor}(y_1 - 1,y_1 + y_2 - 1).
\end{align*}
Independence and translation invariance then imply
\begin{align*}
&P \left( \log Z_{0, \lfloor x_1 + x_2 \rfloor}(0,y_1 + y_2 - 1) \geq z_1 + z_2 \right) \\
 &\hskip100pt \geq P \left( \log Z_{0, \lfloor x_1 \rfloor}(0,y_1 - 1) \geq z_1 \right) P \left( \log Z_{0, \lfloor x_2 \rfloor + x_{1,2}}(0,y_2) \geq z_2 \right).
\end{align*}
If $x_{1,2} = 0$ then, recalling that $\log Z_{\lfloor x_2\rfloor, \lfloor x_2 \rfloor}(u,t) = B_{\lfloor x_2 \rfloor}(u,t)$, we find
\begin{align*}
P \left( \log Z_{0, \lfloor x_2 \rfloor}(0,y_2) \geq z_2 \right) &\geq P \left( \log Z_{0, \lfloor x_2 \rfloor }(0,y_2 - 1) \geq z_2 \right) P\left(\log Z_{\lfloor x_2 \rfloor, \lfloor x_2 \rfloor}(y_2 -1, y_2) \geq 0\right) \\
&= \frac{1}{2}P \left( \log Z_{0, \lfloor x_2 \rfloor}(0,y_2 - 1) \geq z_2 \right).
\end{align*}
Similarly, when $x_{1,2} = 1$ we have
\begin{align*}
P \left( \log Z_{0, \lfloor x_2 \rfloor +1}(0,y_2) \geq z_2\right) &\geq  P \left( \log Z_{0, \lfloor x_2 \rfloor}(0,y_2-1) \geq z_2\right) P \left( \log Z_{0, 1}(0,1) \geq 0\right).
\end{align*}

Setting $C = \max\{\log(2), - \log P \left(\log Z_{0, 1}(0,1) \geq 0\right)\} < \infty$, we find that
\begin{align*}
T(x_1 + x_2, y_1 + y_2, z_1 + z_2) &\leq T(x_1, y_1, z_1) + T(x_2, y_2, z_2) + C.
\end{align*}
$T(x,y,z)$ is therefore subadditive with a bounded correction. Non-negativity and \hyperref[lem:existlb]{Lemma \ref*{lem:existlb}} imply that $T(x,y,z)$ is bounded for $x,y,z$ in a compact subset of its domain. The proof of \cite[Theorem 16.2.9]{Kuczma} shows that we may define a function on $[0,\infty) \times (0,\infty) \times \mathbb{R}$ by
\begin{align*}
J_{s,t}(r) &=  \lim_{x \to \infty} - \frac{1}{x} \log P \left(\log Z_{0, \lfloor xs \rfloor}(0, xt - 1) \geq xr \right)
\end{align*}
and that this function satisfies all of the regularity properties in the statement of the theorem except continuity and monotonicity. Monotonicity in $r$ for fixed $s$ and $t$ follows from monotonicity in the prelimit expression. Convexity and finiteness imply continuity on $(0,\infty)^2 \times \mathbb{R}$ \cite[Theorem 10.1]{Rock}. Moreover, \cite[Theorem 10.2]{Rock} gives upper semicontinuity on all of $[0,\infty) \times (0,\infty) \times \mathbb{R}$. It therefore suffices to show lower semicontinuity at the boundary; namely, we need to show
\begin{align*}
\liminf_{(s',t',r') \to (0,t,r)} J_{s', t'}(r') &\geq J_{0,t}(r).
\end{align*}

Fix $(t, r) \in (0,\infty) \times \mathbb{R}$ and a sequence $(s_k,t_k,r_k) \in [0,\infty) \times (0,\infty) \times \mathbb{R}$ with $(s_k, t_k, r_k) \to (0,t, r)$. Recall that $\log Z_{0,0}(0,t) = B_0(t)$, so we may compute with the normal distribution to find $J_{0,t}(r) = \frac{r^2}{2t}1_{\{r \geq 0\}}$. From this we can see that if $s_k = 0$ for all sufficiently large $k$, we have $J_{s_k, t_k}(r_k) \to J_{0, t}(r)$. We may therefore assume without loss of generality that $s_k > 0$ for all $k$. First observe that if $r \leq 0$, then $J_{0,t}(r) = 0$ and the lower bound follows from non-negativity.

If $r > 0$, we may assume without loss of generality that there exists $c >0$ with $r_k > c$ for all $k$. By \hyperref[lem:JGUE]{Lemma \ref*{lem:JGUE}} in the appendix, for all sufficiently large $k$ we have
\begin{align*}
J_{s_k, t_k}(r_k) &\geq  s_k J_{GUE}\left(\frac{r_k - s_k \log t_k - s_k + s_k\log s_k}{2 \sqrt{t_k s_k}} - 1  \right),
\end{align*}
where $J_{GUE}(r) = 4\int_0^r \sqrt{x(x+2)}dx$. Using this formula and calculus, we find that
\begin{align*}
\lim_{k \to \infty} s_k J_{GUE}\left(\frac{r_k - s_k \log t_k - s_k + s_k \log s_k}{2 \sqrt{t_k s_k}} - 1  \right) &= \frac{r^2}{2t}
\end{align*}
and therefore continuity follows. \hyperref[lem:fe]{Lemma \ref*{lem:fe}} implies that $J_{s,t}(r) = 0$ for $r \leq \rho(s,t)$.
\end{proof}

\begin{remark}
Note that we only address the spatial boundary in the previous result. The reason for this is that the right tail rate function is not continuous at $t = 0$ for any $s > 0$ and $x \in \mathbb{R}$. To see this, we can use the lower bound for $J_{s,t}(r)$ coming from \hyperref[lem:JGUE]{Lemma \ref*{lem:JGUE}}. As $t \downarrow 0$, this lower bound tends to infinity.
\end{remark}
\begin{lemma} \label{lem:sequences}
Fix $(s,t,r) \in (0,\infty)^2 \times \mathbb{R}$. For any sequences $s_n, t_n\in \mathbb{N} \times (0,\infty)$ with $\frac{1}{n}(s_n, t_n) \to (s,t)$ we have
\begin{align*}
J_{s,t}(r) &= \lim_{n \to \infty} - \frac{1}{n} \log P \left( \log Z_{0, s_n }(0, t_n) \geq  n r \right).
\end{align*}
\end{lemma}
\begin{proof}
Fix $\epsilon < \min(s,t)$ and positive. We will assume that $n$ is large enough that the following conditions hold:
\begin{align*}
\Big\lfloor \left(s - \frac{\epsilon}{2}\right)n \Big\rfloor < s_n < \Big \lfloor \left(s + \frac{\epsilon}{2}\right)n \Big\rfloor , \qquad \left(t - \frac{\epsilon}{2}\right)n < t_n < \left(t + \frac{\epsilon}{2}\right) n - 2.
\end{align*}
We have
\begin{align*}
Z_{0, s_n}(0, t_n) &\geq Z_{0, \lfloor (s - \epsilon) n \rfloor}(0, (t - \epsilon)n - 1) Z_{\lfloor (s - \epsilon) n\rfloor, s_n}( (t - \epsilon)n - 1, t_n).
\end{align*}
It follows that
\begin{align*}
&P \left(\log Z_{0, s_n}(0, t_n) \geq n r \right) \\
&\hskip20pt \ge P\left(\log Z_{0, \lfloor (s - \epsilon) n \rfloor}(0, (t - \epsilon)n - 1) \geq nr\right) P\left(\log Z_{0, s_n - \lfloor (s-\epsilon)n \rfloor}(0, t_n - (t - \epsilon)n + 1) \geq 0 \right). \notag
\end{align*}
Call $s(n) = s_n - \lfloor (s - \epsilon) n \rfloor$ and $t(n) = t_n - (t - \epsilon) n + 1$ and divide the interval $(0, t(n))$ into $s(n)$ uniform subintervals. We may bound $Z_{0, s(n)}(0,t(n))$ below by a product of i.i.d. random variables:
\begin{align*}
Z_{0, s(n)}(0,t(n)) \geq \prod_{i=1}^{s(n)} Z_{i-1,i}\left((i-1)\frac{t(n)}{s(n)}, i \frac{t(n)}{s(n)}\right).
\end{align*}
Therefore,
\begin{align*}
P\left(\log Z_{0, s(n)}(0,t(n)) \geq 0 \right) &\geq P\left(\log Z_{0, 1}\left(0,\frac{t(n)}{s(n)}\right) \geq 0 \right)^{s(n)}.
\end{align*}
Notice that $\lim \frac{s(n)}{n} = \epsilon$ and $\lim \frac{t(n)}{n} = \epsilon$, so we may further assume without loss of generality that $\frac{1}{2} < \frac{t(n)}{s(n)} < 2$ for all $n$. We have
\begin{align*}
Z_{0,1}\left(0, \frac{t(n)}{s(n)} \right) \geq Z_{0,1}\left(0, \frac{1}{2}\right)Z_{1,1}\left(\frac{1}{2}, \frac{t(n)}{s(n)}\right) =Z_{0,1}\left(0, \frac{1}{2}\right)e^{B_1\left(\frac{1}{2}, \frac{t(n)}{s(n)} \right)}
\end{align*}
so that
\begin{align*}
P\left(\log Z_{0, 1}\left(0,\frac{t(n)}{s(n)}\right) \geq 0 \right) &\geq \frac{1}{2} P\left(\log Z_{0, 1}\left(0,\frac{1}{2}\right) \geq 0 \right).
\end{align*}
Therefore for $C = \log(2) - \log P\left(\log Z_{0, 1}\left(0,\frac{1}{2}\right) \geq 0 \right)$ and all $\epsilon < \min(s,t)$ we have
\begin{align*}
\limsup - \frac{1}{n} \log P\left( \log Z_{0, s_n}(0,t_n) \geq nr \right) \leq J_{s - \epsilon, t - \epsilon}(r) + \epsilon C
\end{align*}
sending $\epsilon \downarrow 0$  and applying continuity of the rate function gives one inequality. A similar argument gives  the $\liminf$ inequality.
\end{proof}

The next corollary follows from convexity of $J_{s,t}(x)$ in $(s,t,x) \in (0,\infty)^2\times \mathbb{R}$. Details can be found in the first few lines of the proof of \cite[Lemma 4.6]{GS13}.
\begin{lemma}\label{lem:lfreg}
For all $\xi > 0$, $J_{s,t}^*(\xi)$ is convex as a function of $(s,t) \in (0,\infty)^2$.
\end{lemma}

\subsection{Regularity for the stationary right tail rate functions} \label{subsec:statreg}
Next, we turn to regularity for $H_{u,v,s,t}^\theta(x)$ and $G_{a,s,t}^\theta(x)$, which are defined in \hyperref[GHdef]{(\ref*{GHdef})}. We begin with the proof of \hyperref[lem:Hcont]{Lemma \ref*{lem:Hcont}}. This result is the only point in the paper where we directly use the continuity up to the boundary in \hyperref[thm:exist]{Theorem \ref*{thm:exist}}.
\begin{proof}[Proof of Lemma \ref*{lem:Hcont}]
Notice that $\theta t, a\Psi_0(\theta),$ and $\rho(s-b,t + \gamma)$ are bounded for $a,b \in [0,s]$ and $\gamma \in [0,t]$. Using this fact and the formula for $H_{a,b,s,t + \gamma}^\theta(r)$ coming from \hyperref[cor:infcon]{Corollary \ref*{cor:infcon}} and \hyperref[lem:indep]{Lemma \ref*{lem:indep}}, there exists a compact set $K'$ containing $K$ so that for all $r \in K$, $a,b \in [0,s]$, and $\gamma \in [0,t]$
\begin{align*}
H_{a,b,s,t + \gamma}^\theta(r) &= \inf_{x \in K'} \{R_t^\theta \square U_a^\theta(x) + J_{s-b, t + \gamma}(r - x)\}.
\end{align*}
Note that $(a,x) \mapsto R_t^\theta \square U_a^\theta(x)$ is continuous on $[0,s] \times \mathbb{R}$. By \hyperref[thm:exist]{Theorem \ref*{thm:exist}}, for any compact set $K'$ we have joint uniform continuity of $(a, b, \gamma, r,x) \mapsto R_t^\theta \square U_a^\theta(x) + J_{b, t + \gamma}(r-x)$ on the compact set $[0,s]^2 \times [0,t] \times K' \times K'$ and so the result follows.
\end{proof}
The proof of \hyperref[lem:Gcont]{Lemma \ref*{lem:Gcont}} is similar to the proof of \hyperref[lem:Hcont]{Lemma \ref*{lem:Hcont}}, so we omit it. Next, we turn to the proof of \hyperref[lem:Ginf]{Lemma \ref*{lem:Ginf}}, which shows that $G_{a,s,t}^\theta(x)$ tends to infinity locally uniformly near near $a = t$.
\begin{proof}[Proof of Lemma \ref*{lem:Ginf}]
We have
\begin{align*}
G_{a,s,t}^\theta(x) &= \begin{cases}
0 & x \leq - \theta(t-a) + \rho(s, t-a) \\
\displaystyle \inf_{-\theta (t-a) \leq y \leq x - \rho(s,t-a)} \{J_{s,t-a}(x-y) + R_{t-a}^\theta(y)\} & x > - \theta(t-a) + \rho(s, t-a)
\end{cases}
\end{align*}

Fix $\epsilon > 0$. The formula in \hyperref[lem:fe]{Lemma \ref*{lem:fe}} shows that $\rho(s,t-a) \to -\infty$ as $a \uparrow t$, so that for all $x \in \mathbb{R}$ and $a$ sufficiently close to $t, x > - \theta(t-a) + \rho(s, t-a)$. For $a$ sufficiently large that this holds for all $x \in K$, we have
\begin{align*} 
J_{s,t-a}(x-y) + R_{t-a}(y) &\geq \begin{cases}
J_{s,t-a}(x + \theta (t-a) - \epsilon)) \quad &  y \in [-\theta(t-a), -\theta(t-a) + \epsilon] \\
R_{t-a}^\theta(-\theta(t-a) + \epsilon)  & y\in [-\theta(t-a) + \epsilon, x-\rho(s,t-a) - \epsilon] \\ 
R_{t-a}^\theta(x - \rho(s,t-a) - \epsilon)  & y\in[x - \rho(s,t-a) - \epsilon, x - \rho(s,t)] 
\end{cases}
\end{align*} 
By \hyperref[lem:JGUE]{Lemma \ref*{lem:JGUE}}, for all $x \in K$ and $a$ sufficiently large, we have
\begin{align*}
J_{s,t-a}(x) &\geq  s J_{GUE}\left(\frac{x - s \log(t-a) - s (1 - \log(s))}{2 \sqrt{(t-a) s}} - 1  \right)
\end{align*}
Combining this with the exact formula for $R_{t-a}^\theta(x)$ and optimizing the lower bounds over $x \in K$ shows that the infimum over $x \in K$ of the minimum of these three lower bounds tends to infinity, giving the result.
\end{proof}

\section{Technical estimates} \label{sec:estimates}
To reduce the clutter elsewhere in the paper, we collect a number of useful estimates in this appendix.
\subsection{A lower bound on the probability of being large}
\begin{lemma} \label{lem:existlb}
Let $K \subset [0,\infty) \times (0,\infty) \times \mathbb{R}$ be compact. Then there exists $C_K > 0$ so that for all $(x,y,z) \in K$
\begin{align*}
P(\log Z_{0, \lfloor x \rfloor}(0,y) \geq z) &\geq C_K.
\end{align*}
\end{lemma}
\begin{proof}
Since $\lfloor x \rfloor $ takes only finitely many values on any compact set, we may fix $\lfloor x \rfloor$. If $\lfloor x \rfloor = 0$, then $Z_{0, \lfloor x \rfloor}(0,y) = B_0(y)$ and the result follows. For $\lfloor x \rfloor \geq 1$, we bound below by an i.i.d. product: $Z_{0, \lfloor x \rfloor}(0,y)  \geq \prod_{i=0}^{\lfloor x \rfloor -1} Z_{i, i+1}\left(i \frac{y}{\lfloor x \rfloor},(i+1)\frac{y}{\lfloor x \rfloor}\right)$. It follows that
\begin{align}\label{ineq:existlb}
P\left(\log Z_{0, \lfloor x \rfloor}(0,y) \geq z  \right) \geq P\left( \log Z_{0,1}\left(0, \frac{y}{\lfloor x \rfloor}\right) \geq \frac{z}{\lfloor x \rfloor} \right)^{\lfloor x \rfloor}.
\end{align}
Jensen's inequality applied to $\log Z_{0,1}(0,t)$ gives
\begin{align}\label{Jensenlb}
\log Z_{0,1}(0,t) &= \log \int_0^t e^{B_0(u) +B_1(u,t)}du \geq \log(t) + \frac{1}{t}\int_0^t B_0(u) du + \frac{1}{t}\int_0^t B_1(u,t) du,
\end{align}
where $\frac{1}{t}\int_0^t B_0(u) du$ and $\frac{1}{t}\int_0^t B_1(u,t) du$ are i.i.d. mean zero normal random variables with variance $\frac{t}{3}$. Applying this lower bound to the expression in \hyperref[ineq:existlb]{(\ref*{ineq:existlb})} gives the result. \end{proof}

\subsection{Moment estimate for the partition function}
\begin{lemma} \label{momentbound}
Fix $t > 0$, $n \in \mathbb{N}$ and $\xi \in \mathbb{R}$ with $|\xi| > 1$. Then there is a constant $C > 0$ depending only on $\xi$ so that
\begin{align*}
E\left[ Z_{1,n}(0,t)^{\xi} \right] &\leq C \left( \frac{\sqrt{n}}{t}\left(\frac{te}{n}\right)^n\right)^\xi e^{\frac{1}{2}\xi^2 t}.
\end{align*}
\end{lemma}
\begin{proof}
Recall the definition of $A_{n,t}$ in \hyperref[Weyldef]{(\ref*{Weyldef})}.  By Jensen's inequality with respect to the uniform measure on $A_{n,t}$ and Tonelli's theorem we find
\begin{align*}
E\left[ Z_{1,n}(0,t)^{\xi} \right] &=  E\left[\left(\int_{0 < s_1 < \dots < s_{n-1} < t}e^{\sum_i B(s_i, s_{i+1})} ds_1 \dots ds_{n-1}\right)^\xi \right] \\
&\leq |A_{n,t}|^{-1} |A_{n,t}|^\xi \int_{0 < s_1 < \dots < s_{n-1} < t} E[e^{\xi \sum_i B(s_i, s_{i+1})} ]ds_1 \dots ds_{n-1} \\
&= |A_{n,t}|^{\xi} e^{\frac{1}{2}\xi^2 t},
\end{align*}
where we have used independence of the Brownian increments and the moment generating function of the normal distribution to compute the last line. The remainder of the statement of the lemma comes from the identity $|A_{n,t}| = \frac{t^{n-1}}{(n-1)!}$ and Stirling's approximation to $n!$.
\end{proof}

\subsection{Bounds from the GUE connection} \label{GUEsub}
Let $\lambda_{GUE,n}$ be the top eigenvalue of an $n\times n$ GUE random matrix with entries that have variance $\sigma^2 = \frac{1}{4n}$. Then \cite[Theorem 0.7]{Bar01} and \cite{GTW01} give
\begin{align*}
\lambda_{GUE,n} \stackrel{d}{=} \frac{1}{2\sqrt{n}} \max_{0 = u_0 < u_1 < \dots < u_{n-1} < u_n = 1}\left\{ \sum_{i=1}^n B_i(u_{i-1}, u_i)\right\}.
\end{align*}
The right tail rate function of $\lambda_{GUE,n}$ can be computed (\cite[(1.25)]{Ledoux}, \cite{BDG01}) for $r>0$  to be
\begin{align} \label{JGUE}
J_{GUE}(r) = \lim_{n \to \infty} - \frac{1}{n} \log P\left(\lambda_{GUE,n} \geq 1 + r  \right) =  4 \int_0^r \sqrt{x(x+2)}dx
\end{align}

\begin{lemma}\label{lem:JGUE} Suppose that $r,s,t>0$ and $(s_n, t_n, r_n) \in \mathbb{N} \times (0,\infty) \times \mathbb{R}$ satisfy $n^{-1}(s_n,t_n,r_n) \to (s,t,r)$. If $r - s \log t- s + s\log s > 2 \sqrt{t s}$, then
\begin{align*}
\liminf_{n \to \infty} - \frac{1}{n} \log P\left(\log Z_{0, s_n}(0,t_n) \geq r_n \right) &\geq s J_{GUE}\left(\frac{r - s \log t - s + s \log s)}{2 \sqrt{t s}} - 1\right)
\end{align*}
and if $r + s \log t +  s - s\log s > 2 \sqrt{t s}$, then
\begin{align*}
\liminf_{n \to \infty} - \frac{1}{n} \log P\left(\log Z_{0, s_n}(0,t_n) \leq - r_n \right) &\geq s J_{GUE}\left(\frac{r + s \log t + s - s \log s}{2 \sqrt{t s}} - 1\right).
\end{align*}
\end{lemma}
\begin{proof}
Recall the definition of $A_{n,t}$ in \hyperref[Weyldef]{(\ref*{Weyldef})} and observe that
\begin{align*}
|A_{n+1,t}| = \frac{t^n}{n!} \leq \frac{1}{\sqrt{2\pi n}}\left(\frac{te}{n}\right)^n.
\end{align*}
Using this fact and bounding $Z_{0,n}(0,t)$ as defined in \hyperref[polydefgr]{(\ref*{polydefgr})} above with the maximum value of the Brownian increments, we obtain
\begin{align*}
\log Z_{0, s_n}(0, t_n) &\leq \log\left( \frac{1}{\sqrt{2\pi s_n}} \left(\frac{t_n e}{s_n}\right)^{s_n} \right) + \max_{0 = u_0 < u_1 < \dots < u_{s_n} = t_n}\left\{ \sum_{i=0}^{s_n-1} B_i(u_i, u_{i+1})\right\} \\
&\stackrel{\tiny{d}}{=} \log\left( \frac{1}{\sqrt{2\pi s_n}} \left(\frac{t_n e}{s_n}\right)^{s_n} \right) + 2\sqrt{t_n s_n} \lambda_{GUE, s_n}
\end{align*}
The result then follows from the inequality
\begin{align*}
P\left(\log Z_{0, s_n}(0,t_n) \geq r_n \right) \leq P\left(\lambda_{GUE, s_n} \geq \frac{r_n - s_n\log t_n - s_n + s_n \log s_n }{2\sqrt{t_n s_n}} - \frac{1}{2\sqrt{t_ns_n}}\log\left(\frac{1}{\sqrt{2\pi s_n}}\right) \right).
\end{align*}
The proof of the second bound follows a similar argument: we bound the partition function below with the minimum of the Brownian increments, apply the upper bound from Stirling's approximation to $n!$, and appeal to Brownian reflection symmetry.
\end{proof}
\begin{lemma}\label{lem:HGUE}
Fix $\epsilon > 0$ and let $s \in \mathbb{N}$  and $t_n  = O(n^\alpha)$ for some $\alpha < 1$. Then there exist $c,C > 0$ so that
\begin{align*}
P\left(\max_{0 = u_0 < u_1 < \dots < u_{s - 1} < u_s =  t_n}\left\{ \sum_{i=0}^{s - 1}B_i(u_i, u_{i+1})\right\} \geq n \epsilon \right) &\leq Ce^{-cn^{2 - \alpha}}
\end{align*}
\end{lemma}
\begin{proof} Large deviation estimates for largest eigenvalues give the result. For example, by  \cite[(2.7)]{Ledoux}, there exist $C,c>0$ such that
\begin{align*}
P\left(\max_{0 = u_0 < u_1 < \dots < u_{s - 1} < u_s =  t_n} \left\{ \sum_{i=0}^{s - 1}B_i(u_i, u_{i+1}) \geq n \epsilon \right\}\right) &= P\left(\lambda_{GUE,s} \geq \frac{n}{\sqrt{t_n}}\frac{\epsilon}{\sqrt{s}} \right) \leq Ce^{-cn^{2 - \alpha}}.
\end{align*}
\end{proof}

\subsection{Upper tail coarse graining estimate}
\begin{lemma} \label{uppercoarse1} Fix $a \in [0,t)$ and $\epsilon > 0$. Then for $\nu < \min(\epsilon, t-a)$
\begin{align*}
&P\left(\log n \int_a^{a + \nu} \frac{Z_0^\theta(nu)}{Z_0^\theta(n a)} \cdot \frac{Z_{1,\lfloor ns \rfloor}(nu,nt)}{Z_{1,\lfloor ns \rfloor}(n a,nt)} du \geq n \epsilon \right)  \leq \exp\left\{-n \frac{1}{4} \left(\frac{\epsilon - \theta \nu}{\sqrt{\nu}}\right)^2 +o(n)\right\}.
\end{align*}
\end{lemma} 

\begin{proof}
By \hyperref[gensubmult]{(\ref*{gensubmult})}, we have for all $u \in(a, a+ \nu)$
\begin{align*}
Z_{1,1}(na,nu)^{-1}Z_{1, \lfloor ns \rfloor}(nu,nt)^{-1} &\geq Z_{1,\lfloor ns \rfloor}(na,nt)^{-1}
\end{align*}
so it follows that
\begin{align*}
&P\left(\log n \int_a^{a + \nu} \frac{Z_0^\theta(nu)}{Z_0^\theta(n a)}\frac{Z_{1,\lfloor ns \rfloor}(nu,nt)}{Z_{1,\lfloor ns \rfloor}(n a,nt)} du \geq n \epsilon \right)  \notag \\
&\hskip35pt \leq P\left(\log n \int_a^{a + \nu} \frac{Z_0^\theta(nu)}{Z_0^\theta(n a)} Z_{1,1}(na,nu)^{-1} du \geq n \epsilon \right) \notag \\
&\hskip35pt= P\left(\log n \int_a^{a + \nu}e^{\theta n(u-a) - B(na,nu) - B_1(na,nu)} du \geq n \epsilon \right) \\
&\hskip35pt \leq P\left( \max_{0 \leq u \leq 1}\left\{B(u) + B_1(u) \right\} \geq \sqrt{n} \left(\frac{\epsilon - \theta \nu}{\sqrt{\nu}}\right) - \frac{\log(n \nu)}{\sqrt{n \nu}} \right),
\end{align*}
where the last inequality comes from Brownian translation invariance, symmetry, and scaling. Recall that $B + B_1$ has the same process level distribution as $\sqrt{2} B$. The result follows from the reflection principle.
\end{proof}

\subsection{Left tail error bound}
\begin{lemma}\label{lem:normtail}
Take sequences $t_n, s_n, r_n$ such that there exist $a,b > 0$ with $a < t_n < b$, $r_n \to r > 0$ and $s_n \in \mathbb{N}$ satisfies $s_n \log (s_n) = o(n)$. Then there exist constants $c,C>0$ such that
\begin{align*}
P\left(\log Z_{0,s_n}(0,t_n) \leq - n r_n \right) &\leq C e^{- c n^2}.
\end{align*}
\end{lemma}
\begin{proof}
We have $Z_{0,s_n}(0,t_n)\geq \prod_{i=0}^{s_n -1}Z_{i, i+1}\left(i \frac{t_n}{s_n}, (i+1) \frac{t_n}{s_n}\right)$ where the $Z_{i, i+1}\left(i \frac{t_n}{s_n}, (i+1) \frac{t_n}{s_n}\right)$ are i.i.d..  As above in \hyperref[Jensenlb]{(\ref*{Jensenlb})}, there exist i.i.d. random variables $X_i \sim N\left(\log \left( \frac{t_n}{s_n} \right), \frac{2 t_n}{3 s_n}\right)$ with $Z_{i, i+1}\left(i \frac{t_n}{s_n}, (i+1) \frac{t_n}{s_n}\right) \geq X_i$. It follows that
\begin{align*}
P\left(\log Z_{0,s_n}(0,t_n) \leq - n r_n\right) \leq P\left(\sum_{i=0}^{s_n -1} X_i \leq - n r_n\right) = P\left( N\left(0, 1\right) \geq n \frac{r_n}{\sqrt{3 t_n}} + \frac{s_n}{\sqrt{3 t_n}}\log\left(\frac{t_n}{s_n}\right) \right).
\end{align*}
Recall that $\frac{r_n}{\sqrt{3 t_n}} + \frac{s_n}{n\sqrt{3 t_n}}\log\left(\frac{t_n}{s_n}\right)$ is a bounded sequence and without loss of generality is bounded away from zero. The result follows from normal tail estimates.
\end{proof}

\end{appendices}

\end{document}